\documentclass{aims}
\def\currentvolume{}
 \def\currentissue{}
  \def\currentyear{}
   \def\currentmonth{}
    \def\ppages{}
     \def\DOI{}

\usepackage{fancyheadings}
\usepackage{graphicx} 
\usepackage{epsfig} 
\usepackage{amsmath} 
\usepackage{amssymb} 

\usepackage{psfrag,xspace}
\usepackage{mathtools} \usepackage{enumerate,pdfsync}
 
\usepackage{subcaption} \usepackage{tikz} \usepackage{tikz-cd}
\usepackage{commands}

\usepackage[mathscr]{euscript}
 \usepackage[colorlinks=true]{hyperref}

\renewcommand{\th}{s}
\theoremstyle{plain}
\newtheorem{thm}{Theorem}[section]

\newtheorem{lem}[thm]{Lemma}
\newtheorem{prop}[thm]{Proposition}

\theoremstyle{definition}

\newtheorem{rem}[thm]{Remark}

\usepackage{marginnote}
\newcounter{changescounter}

\title[A coordinate-free theory
of virtual holonomic constraints]%
{A coordinate-free theory \\ of virtual holonomic constraints}

\author[Luca Consolini, Alessandro
  Costalunga, and Manfredi Maggiore]{}

\email{luca.consolini@unipr.it}
\email{alessandro.costalunga@studenti.unipr.it}
\email{maggiore@ece.utoronto.ca}


\thanks{Some of the ideas of this paper appeared in preliminary form
  in~\cite{ConCos15}} \thanks{$^*$ Corresponding author}

\begin{document}
\maketitle
 \thispagestyle{fancy} \cfoot{Published in {\em Journal of Geometric
     Mechanics}, vol. 10, no. 4, pp. 467--502, December 2018}
\headrulewidth 0 pt
\chead{}\lhead{}\rhead{}

% Enter the first author's name and address:
\centerline{\scshape Luca Consolini, Alessandro Costalunga,}
\medskip
{\footnotesize
% please put the address of the first author
 \centerline{Dipartimento di Ingegneria dell'Informazione, Universit\`a di Parma
   } 
\centerline{Parco Area delle Scienze 181/a, 43124 Parma,
   Italy}
}
\medskip

\centerline{\scshape and Manfredi Maggiore$^*$}
\medskip
{\footnotesize
 % please put the address of the second and third author
  \centerline{Department of Electrical and Computer Engineering,
    University of Toronto} 
\centerline{10 King's College Road, Toronto, Ontario, M5S 3G4, Canada}
}

%% \title{A coordinate-free theory of virtual holonomic constraints}

%% \author[UniPr]{Luca Consolini}\ead{luca.consolini@unipr.it}, % Add the
%% \author[UniPr]{Alessandro
%%   Costalunga}\ead{alessandro.costalunga@studenti.unipr.it}, %
%% \author[UniTo]{Manfredi Maggiore}\ead{maggiore@control.utoronto.ca} %

%% \address[UniPr] {Dipartimento di Ingegneria dell'Informazione,
%%   Universit\`a di Parma, Parco Area delle Scienze 181/a, 43124 Parma,
%%   Italy} \address[UniTo]{Department of Electrical and Computer
%%   Engineering, University of Toronto, 10 King's College Road, Toronto,
%%   Ontario, M5S 3G4, Canada.}

%% \begin{keyword}
%% Mechanical control systems; virtual holonomic constraints; inverse
%% Lagrangian problem; geometric mechanics.
%% \end{keyword}

%%%%%%%%%%%%%%%%%%%%%%%%%%%%%%%%%%%%%%%%%%%%%%%%%%%%%%%%%%%%%%%%%%%%%%%%%%%%%%%%
%\end{frontmatter}

\bigskip

% The name of the associate editor will be entered by an editorial staff
% "Communicated by the associate editor name" is not needed for special issue.
% \centerline{(Communicated by the associate editor name)}

\begin{abstract}
This paper presents a coordinate-free formulation of virtual holonomic
constraints for underactuated Lagrangian control systems on Riemannian
manifolds.  It is shown that when a virtual constraint enjoys a
regularity property, the constrained dynamics are described by an
affine connection dynamical system. The affine connection of the
constrained system has an elegant relationship to the Riemannian
connection of the original Lagrangian control system. Necessary and
sufficient conditions are given for the constrained dynamics to be
Lagrangian. A key condition is that the affine connection of the
constrained dynamics be metrizable. Basic results on metrizability of
affine connections are first reviewed, then employed in three examples
in order of increasing complexity. The last example is a double
pendulum on a cart with two different actuator configurations. For
this control system, a virtual constraint is employed which confines
the second pendulum to within the upper half-plane.
\end{abstract}

%% \begin{frontmatter}

\section{Introduction}\label{sec:intro}

A virtual holonomic constraint (VHC) for a Lagrangian control system
is a collection of relations among the configuration variables of the
system that can be made invariant via feedback control.  The precise
meaning of this terminology is clarified in what follows, but the key
idea is to emulate via feedback control the presence of a holonomic
constraint in the Lagrangian control system. By appropriate design of
the VHC, the constrained system may display useful properties.

The notion of VHC can be traced back to early twentieth century work
of P. Appell in~\cite{appell1911} and H. Beghin in~\cite{Beghin22},
but it has emerged prominently in the last fifteen years as a tool for
control of biped robots (see,
e.g.,~\cite{nakanishi-2000,PleGriWesAbb03,WesGriKod03,WesGriCheChoMor07,CheGriShi08}),
and as an approach to motion planning for general robotic systems
(e.g.,~\cite{ShiPerWit05,ShiRobPerSan06,ShiFreGus10,FreRobShiJoh08}).
In VHC-based robot control, the motion one wants to induce is
represented implicitly in terms of constraints on the robot's
configuration variables, and the control loop is designed to
asymptotically stabilize a subset of the state space, the so-called
constraint manifold.  This control philosophy stands in contrast to
the standard technique of parametrizing a desired motion by time, and
then stabilizing the resulting reference signals.  The VHC control
paradigm has proved particularly effective in inducing complex
behaviours in underactuated robots, and gives rise to a feedback loop
that is intrinsically robust because it is not driven by any exogenous
signal.

For biped robots, Grizzle and collaborators (see,
e.g.,~\cite{WesGriCheChoMor07}) defined VHCs in terms of invariance of
a submanifold of the state space.  The paper~\cite{6286994}
generalized Grizzle's notion of VHC to mechanical control systems
whose generalized coordinates are linear displacements or angles
(i.e., systems whose configuration manifold is a generalized cylinder)
and whose degree of underactuation is equal to one. For this class of
systems, the authors of~\cite{nolcos_2013_lagr,MohMagCon17} showed
that, generically, the constrained dynamics in the presence of a VHC
do {\em not} possess a Lagrangian structure. They then gave necessary
and sufficient conditions for a Lagrangian structure to exist.
The theory of~\cite{6286994,MohMagCon17} does not handle mechanical
systems whose configuration space is not a generalized cylinder, or
whose degree of underactuation is greater than one. To illustrate, the
configuration manifold of a rigid body is $\mathsf{SE}(3)$, a manifold
that cannot be handled by the theory
in~\cite{6286994,MohMagCon17}. Similarly, a double pendulum on a cart
has degree of underactuation two, which again is not contemplated in
the theory of~\cite{6286994,MohMagCon17}.

{\bf Main contributions.}  This paper generalizes the theory
in~\cite{6286994,MohMagCon17} by presenting a coordinate-free
formulation of VHCs for arbitrary configuration manifolds and
arbitrary degrees of underactuation.  We give a new geometric
definition of VHC, and define a regularity property of VHCs in terms
of transversality of two subbundles. We show that a regular VHC
induces on the constraint manifold an affine connection, the so-called
{\em induced connection.}  In the absence of a potential function,
orbits of the constrained dynamics are geodesics of this induced
connection.  We give an explicit characterization of the constrained
dynamics in coordinates with formulas for the Christoffel symbols of
the induced connection.  We show that the problem of determining
whether or not the constrained dynamics are Lagrangian amounts in
great part to determining whether or not the induced connection is
metrizable, i.e., it is Riemannian for a suitable metric. Leveraging
this insight, and using existing results from the theory of affine
connections, we give conditions for the existence of a Lagrangian
structure for the constrained dynamics arising from a regular
VHC. These conditions are applicable to Lagrangian control systems
with arbitrary degree of underactuation. In the special case when the
subbundle associated with the control accelerations is orthogonal to
the constraint manifold, the constrained dynamics are always
Lagrangian, and we show that they coincide with the dynamics one would
obtain in the presence of an ideal holonomic constraint. Thus, the
classics mechanics notion of holonomic constraint is a special
case of our theory.  For systems with underactuation degree one, our
results provide an elegant geometric insight for the results
in~\cite{MohMagCon17}.

The focus of this paper is on the case when all control
inputs are used to enforce the VHC, so that the constrained dynamics
are unforced. In the more general case when the constrained dynamics
are forced, the question of existence of a Lagrangian structure for
the constrained dynamics turns into the more general question of
feedback equivalence of the constrained dynamics to a Lagrangian
control system. A local version of the latter question has been investigated
in~\cite{RicRes10} for general control systems on smooth manifolds
near zero velocity points. 

{\bf Organization of the paper.}  In Section~\ref{sec:preliminaries}
we review concepts of Riemannian geometry, and the definition
from~\cite{geomControlBullo} of a Lagrangian (control) system on a
Riemannian manifold.  Section~\ref{sec:background_vhc} reviews the
definition of VHC from~\cite{6286994} and the Lagrangian properties of
the constrained dynamics from~\cite{MohMagCon17}, valid for the case
of systems with degree of underactuation one.
Section~\ref{sec:coordinate_free} formulates a new coordinate-free
theory of VHCs, characterizing the regularity of VHCs in terms of
transversality of the constraint manifold and the distribution induced
by control forces.  It is shown that a VHC induces an affine
connection on the constraint manifold, and this connection is then
used to characterize the constrained dynamics.  In
Section~\ref{sec:ILP} we give necessary and sufficient conditions
under which the constrained dynamics are Lagrangian. We also treat the
special case when the distribution induced by control forces is
orthogonal to the constraint manifold. In
Section~\ref{sec:metrizability} we give a tutorial overview of
holonomy groups and results on metrizability of affine connections,
treating the special cases of flat connections, of simply connected
constraint manifolds, and one and two-dimensional constraint
manifolds. Here we show that the results of
Section~\ref{sec:background_vhc} are a special case of the general
theory of this paper. In Section~\ref{sec:examples}, we present three
examples illustrating the theory, in order of increasing
complexity. The last example is a double pendulum on a cart with a VHC
that constrains the angle of the second pendulum to be a function of
the angle of the first pendulum, in such a way that the second
pendulum is always confined to the upper half plane.

\section{Preliminaries}\label{sec:preliminaries}
In this section we present the notation used in this paper, review
notions of Riemannian geometry, and review the definition of a
Lagrangian (control) system on a Riemannian manifold. All results are
found in~\cite{book:598491,doCarmo,Boo86,geomControlBullo}.

{\bf Smooth manifolds.} If $\M$ is a smooth manifold, we denote by
$C^\infty(\M)$ the ring of smooth real-valued functions on $\M$, by
$\X(\M)$ the set of smooth vector fields on $\M$, and by $\Omega(\M)$
the set of smooth one-forms on $\M$. The tangent space to $\M$ at $p
\in \M$ is denoted by $T_p \M$, and it dual, the cotangent space, is
denoted by $T_p^\star \M$.  We denote by $T\M$ and $T^\star \M$ the
tangent and cotangent bundles of $\M$, and by $\pi: T\M \to \M$ the
natural projection on $T\M$. An element of $T\M$ will be denoted by
$v_q$, with the understanding that $v_q \in T_q \M$. If $\N$ is a
submanifold of $\M$, $T\M|_\N$ denotes the restriction of $T\M$ to $\N$,
defined as $T\M|_\N = \bigcup_{p \in \N} T_p \M$.

If $(U,(x^1,\ldots,x^n))$ is a coordinate chart of $\M$, for each $p
\in U$ the basis for $T_p \M$ induced by the chart is denoted by
$\der{1}{p},\ldots, \der{n}{p}$. The vector fields
$\{\de{1},\ldots,\de{n}\}$ form a {\em local frame} for $T\M$.
If $F:\M \to \N$ is a smooth function between smooth manifolds and
$p\in \M$, we let $F_p :=F(p)$, and we denote by $dF_p : T_p \M \to
T_{F(p)} \N$ the differential of $F$ at $p$.
If $F :U \to V $ is a smooth function and $U \subset \Re^n$, $V
\subset \Re^m$ are open sets, we denote by $\partial_{x^i} F$ the
partial derivative $F$ with respect to its $i$-th argument. The
notation $\partial^2_{x^i x^j} F$ indicates second-order partial
differentiation with respect to the $i$-th and $j$-th argument.  More
generally, if $U \subset \Re^n$ is an open set and $F : U \to F(U)
\subset M$ is smooth, then we denote $\partial_{x^i} F:= d F_x
(\diff{x^i})$, where $\diff{x^i}$ denotes the $i$-th natural basis
vector of $T_x \Re^n$. In the special case of a function of one
variable in $\Re$ or $\Se^1$, we let $F'(x) :=\partial_x F$ and
$F''(x):=\partial^2_x F$.

A smooth function $h: \M \to \Re^k$ is a {\em submersion} if $\rank
dh_p = k$ for all $p \in \M$.  If $\rank dh_p=k$ for all $p \in
h^{-1}(0)$, then we say that $0$ is a {\em regular value of $h$}.  If
$f \in C^\infty(\M)$ and $X \in \X(\M)$, the {\em Lie derivative} of
$f$ along $X$ is the smooth function $X(f) \in C^\infty(\M)$ defined
as $p \mapsto X(f)(p) := df_p(X(p))$. If $X,Y \in \X(\M)$, $[X,Y] \in
\X(\M)$ denotes the {\em Lie Bracket} of $X$ and $Y$.

{\bf Riemannian manifolds and connections.} A {\em Riemannian
  manifold} is a pair $(\M,g)$, where $\M$ is a smooth manifold, and
$g: T\M \times T\M \to \Re$, the {\em Riemannian metric}, is a smooth
function such that, for each $p \in \M$, $g_p$ is a bilinear form $T_p
\M \times T_p \M \to \Re$ which is symmetric and positive definite,
i.e., for each $v_p, w_p \in T_p \M$, $g_p(v_p,w_p) = g_p(w_p,v_p)$,
and the function $v_p \mapsto g_p(v_p,v_p)$ is positive
definite. Thus, $g_p$ is an inner product on $T_p\M$ which varies
smoothly with $p$. In the language of tensors, $g$ is a type $(0,2)$
symmetric and positive definite tensor field on $M$.  A Riemannian metric
induces two maps. The {\em flat} map is the function $T \M \to T^\star
\M$, $X \mapsto X^\flat$, defined as $X^\flat(Y) = g(X,Y)$ for all $Y
\in T\M$. The {\em sharp} map is the function $T^\star \M \to T\M$,
$\omega \mapsto \omega^\sharp$, defined uniquely through the identity
$\omega(X) = g(\omega^\sharp,X)$ for all $X \in T\M$.  Given a
function $f \in C^\infty(\M)$, $\grad f: \M \to T\M$ is the smooth
vector field defined as $\grad f := df^\sharp$.

An {\em affine connection} on $\M$ is a smooth function $\nabla :
\X(\M) \times \X(\M) \to \X(\M)$, $(X,Y) \mapsto \nabla_X Y$
satisfying the following properties:
\begin{equation}\label{eq:affine_connection_defn}
\begin{aligned}
& \nabla_{fX + g Y} Z = f \nabla_X Z + g \nabla_Y Z, \\
& \nabla_X(Y+Z) = \nabla_X Y + \nabla_X Z, \\
& \nabla_X (f Y) = f \nabla_X Y + X(f) Y,
\end{aligned}
\end{equation}
for any $f,g \in C^\infty(\M)$ and $X,Y, Z \in \X(\M)$. The vector
field $\nabla_X Y$ is called the {\em covariant derivative} of $Y$ in
the direction of $X$. The covariant derivative of vector fields
induces a covariant derivative of tensor fields, also denoted
$\nabla$, enjoying the properties listed in~\cite[Lemma
  4.6]{book:598491}. Among them, we mention the following. If $F$ is a
tensor field on $\M$ of type $(0,s)$, and $X\in \X(\M)$, then $\nabla_X
F$ is a type $(0,s)$ tensor field satisfying the following identity
\begin{equation}\label{eq:covariant_derivative:tensor}
(\nabla_X F)(Y_1,\ldots,Y_s) = X(F(Y_1,\ldots Y_s)) -\sum_{j=1}^s
F(Y_1,\ldots,\nabla_X Y_j,\ldots,Y_s),
\end{equation}
for all $Y_1,\ldots,Y_s \in \X(\M)$. The {\em total covariant
  derivative} of a type $(0,s)$ tensor field $F$ is the type $(0,s+1)$
tensor field $\nabla F$ given by
\begin{equation}\label{eq:total_covariant_derivative}
\nabla F (X,Y_1,\ldots,Y_s) = \nabla_X F(Y_1,\ldots,Y_s), \ \text{ for
  all } X,Y_1,\ldots,Y_s \in \X(\M).
\end{equation}
The connection $\nabla$ is {\em symmetric} (or
torsionless) if
\[
(\forall X,Y \in \X(\M)) \ \nabla_X Y - \nabla_Y X = [X,Y].
\]
The connection $\nabla$ is {\em compatible with $g$} if
\begin{equation}\label{eq:compatibility}
(\forall X,Y,Z \in \X(\M)) \ X( g(Y,Z) ) = g(\nabla_X Y,Z) +
  g(Y,\nabla_X Z),
\end{equation}
or, in terms of the total covariant derivative,
\begin{equation}\label{eq:compatibility:total_der}
\nabla g =0.
\end{equation}
The Fundamental Lemma of Riemannian Geometry (e.g.,~\cite{book:598491})
states that there is a unique affine connection $\nabla$ on a
Riemannian manifold $(\M,g)$ with the property of being symmetric and
compatible with $g$. This connection is called the {\em Riemannian
  connection} or the {\em Levi-Civita connection} of $g$.

The covariant derivative $\nabla_X Y$ may be viewed as a
differentiation of the vector field $Y$ along $X$. If $\gamma(t)$ is a
smooth curve on $\M$ and $Y \in \X(\M)$, the restriction of $Y$ to
$\gamma(t)$, $V(t) :=Y(\gamma(t))$, is a vector field along
$\gamma$. The derivative $D_t V:=\nabla_{\dot \gamma} Y$ is called the
{\em covariant derivative of $V$ along $\gamma$}. Although the
definition just given relies on expressing $V$ as the restriction to
$\gamma$ of a vector field $Y$ on $\M$, $D_t V$ does not depend on the
values of $Y$ outside of $\gamma(t)$, in that {\em any} smooth
extension of $V$ outside of $\gamma$ gives the same value of $D_t
V$. The geometric intuition of the notion of covariant derivative is
as follows (see, e.g.,~\cite[Chapter VII, Section 2]{Boo86}). If $\M$
is an embedded submanifold of $\Re^n$ with Riemannian metric induced
from an inner product on $\Re^n$, $D_t V$ is the orthogonal projection
of the time derivative of $Y(\gamma(t))$ onto the tangent space
$T_{\gamma(t)} \M$. Thus, roughly speaking, $D_t V$ measures how much
the vector field $V(t)$ turns as seen from the point of view of
$\M$. In essence, covariant derivatives embody the notion of
acceleration of a curve. More precisely, the {\em acceleration} of a
curve $\gamma$ on $\M$ is the vector field $D_t \dot \gamma$ along
$\gamma$, and $\gamma$ is called a {\em geodesic} of $\nabla$ if its
acceleration is zero, i.e., $\nabla_{\dot \gamma} \dot \gamma(t) = D_t
\dot \gamma(t) \equiv 0$. We remark that this definition of geodesic
curve does not require $\nabla$ to be a Riemannian connection.

{\bf Coordinate representation of the covariant derivative.} In
coordinates, covariant derivatives associated with a Riemannian
connection take on a familiar form, which we now review. Consider a
coordinate chart $(U, (x^1,\ldots,x^n))$ on $\M$ and the associated
local frame $\{\de{1},\ldots,\de{n}\}$ for $T\M$. Let $X_i:=\de{i}$.
Given an affine connection $\nabla$ on $\M$, not necessarily
Riemannian, the {\em Christoffel symbols} of $\nabla$ associated with
the local frame $\{X_1,\ldots, X_n\}$ are the $n^3$ functions
$\Gamma_{ij}^k$ in $C^\infty(U)$ that are coefficients of the
expansion of $\nabla_{X_i} X_j$ in the local frame
$\{X_1,\ldots,X_n\}$, i.e.,
\[
\nabla_{X_i} X_j = \sum_{k=1}^n \Gamma_{ij}^k X_k.
\]
One can show that if $\nabla$ is symmetric, then $\Gamma_{ij}^k =
\Gamma_{ji}^k$.  Whether or not $\nabla$ is symmetric, if $Y,Z \in
\X(\M)$, expanding $Y = \sum_i y_i X_i$, $Z = \sum_i z_i X_i$, with
$y_i,z_i \in C^\infty(U)$, the covariant derivative $\nabla_X Y$
can be computed through the formula
\begin{equation}\label{eq:connection_computation}
\nabla_Y Z = \sum_k \left( Y(z_k) + \sum_{i,j} \Gamma_{ij}^k y_i z_j
\right) X_k,
\end{equation}
where $Y(z_k)$ is the Lie derivative of $z_k$ along $Y$. The
acceleration of a smooth curve $\gamma: I \to \M$, $I \subset \Re$, can
be computed as follows. Letting $\gamma^i(t):=x^i(\gamma(t))$ denote
the $i$-th component of the coordinate representation of $\gamma$, we
have
\begin{equation}\label{eq:acceleration}
\nabla_{\dot \gamma} \dot \gamma = \sum_k \left( \ddot \gamma^k +
\sum_{i,j} \Gamma_{ij}^k \dot \gamma^i \dot \gamma^j \right) X_k.
\end{equation}
We see that, in local coordinates, geodesics are solutions of the
system of second-order differential equations
\[
\ddot \gamma^k = - \sum_{i,j} \Gamma_{ij}^k \dot \gamma^i \dot
\gamma^j, \ k=1,\ldots, n.
\]
If $\nabla$ is Riemannian, the Christoffel symbols may be computed
using a matrix representation of the metric $g$. Using again the local
frame $\{X_1,\ldots, X_n\}$, let
$g_{ij}(p):=g(\der{i}{p},\der{j}{p})$, and let $g^{kl}$ be $(k,l)$-th
element of the inverse of the matrix $(g_{ij})$. Then,
\begin{equation}\label{eq:Christoffel_Riemannian}
\Gamma_{ij}^k = \frac{1}{2} \sum_l g^{kl} \big( X_i(g_{jl}) +
X_j(g_{il}) - X_l(g_{ij}) \big).
\end{equation}
{\bf Lagrangian control systems on manifolds.} Having reviewed basic
notions of Riemannian geometry, we are ready to present the class of
mechanical systems considered in this paper. The definitions below are
adapted from~\cite{geomControlBullo}.
\begin{definition}[Lagrangian system] \label{defn:Lagrangian_system}
A {\em Lagrangian system} is a triple $(\Q,g,P)$, where $(\Q,g)$ is an
$n$-dimensional Riemannian manifold called the {\em configuration
  manifold}, and $P : \Q \to \Re$ is a smooth function called the {\em
  potential function}. The triple $(\Q,g,P)$ is also called a {\em
  Lagrangian structure}.  A smooth curve $q : I \to \Q$, where $I$ is
an open interval in $\Re$, is a {\em base integral curve} of the
Lagrangian system if
\begin{equation}\label{eq:Lagrangian_system}
\nabla_{\dot q(t)} \dot q(t) = - \grad P(q(t)),
\end{equation}
for all $t \in I$.
\end{definition}
For each $q_0 \in \Q$ and each $v_{q_0} \in T_{q_0} \Q$, there exists
a unique maximal base integral curve $q(t)$ such that $q(0)=q_0$ and
$\dot q(0)=v_{q_0}$, where maximality is defined in the same way as
for integral curves of vector fields.  We will call $q(t)$ the {\em
  maximal base integral curve of~\eqref{eq:Lagrangian_system} with
  initial condition $(q_0,v_{q_0})$}.

Base integral curves have the property of being extremizers of the
action functional $\int_I L(q(t),\dot q(t)) dt$, $I \subset \Re$,
where $L : T\Q \to \Re$ is the Lagrangian function defined as
\begin{equation}\label{eq:Lagrangian}
L (q,\dot q):=\frac 1 2 g_q (\dot q,\dot q) - P(q).
\end{equation}

In Lagrangian systems, controls appear by way of forces. On Riemannian
manifolds, forces are modelled as one-forms because, under coordinate
changes, they transform like the components of one-forms
(see~\cite{geomControlBullo}). Thus, a force on $\Q$ is a one-form $F
\in \Omega(\Q)$. The corresponding vector field $F^\sharp \in \X(\Q)$
can be thought of as the portion of the acceleration due to $F$. With
this in mind, we proceed to the definition of a Lagrangian control
system.
\begin{definition}[Lagrangian control system] \label{defn:Lagrangian_control_system}
A {\em Lagrangian control system} is a quadruple $(\Q,g,P,\sF)$, where
$(\Q,g,P)$ is a Lagrangian system and $\sF = \{F^1,\ldots, F^m\}$,
$F^i \in \Omega(\Q)$, are called the {\em control forces}.  A curve $q
:I \to \Q$, where $I$ is an open interval in $\Re$, is a {\em base
  integral curve} of the Lagrangian control system if there exist $m$
smooth functions $\tau_i : I \to \Re$ such that
\begin{equation}\label{eq:Lagrangian_control_system}
\nabla_{\dot q(t)} \dot q (t) = -\grad P(q(t)) + \sum_{i=1}^m
(F^i)^\sharp_{q(t)} \tau_i(t),
\end{equation}
for all $t \in I$.
\end{definition}
Let $\tau^\star = (\tau_1^\star,\ldots,\tau_m^\star): T\Q \to \Re^m$
be a smooth map.  Then, for each $q_0 \in \Q$ and each $v_{q_0}
\in T_{q_0} \Q$, there exists a unique maximal solution $q: I \to \Q$
of~\eqref{eq:Lagrangian_control_system} with $\tau_i(t)=
\tau_i^\star(q(t),\dot q(t))$, $i=1,\ldots,m$. We call it the {\em
  maximal base integral curve of~\eqref{eq:Lagrangian_control_system}
  with feedback $\tau=\tau^\star(q,\dot q)$ and initial condition
  $(q_0,v_{q_0})$}.

It is shown in~\cite{geomControlBullo} that the equations of motion of
the Lagrangian system in~\eqref{eq:Lagrangian_system} can be
equivalently expressed as a smooth vector field on $T \Q$. Similarly,
the equations of motion of the Lagrangian control system
in~\eqref{eq:Lagrangian_control_system} have an equivalent
representation as a smooth control-affine system with state space
$T\Q$. In order to define such
  control-affine system, we need the notion of {\em vertical lift} at
  a point $v_q \in T\Q$ (\cite{geomControlBullo}), the linear map
  $\vlft_{v_q} : T_q \Q \to T_{v_q} T\Q$ defined by
\begin{equation}\label{eq:vlft}
\vlft_{v_q} (X_q) = \frac{d}{dt}\Big|_{t=0} (v_q + t X_q).
\end{equation}
The vertical lift of a vector field $X \in \X(\Q)$ is the vector field
$\vlft(X)$ on $T\Q$ defined by $\vlft(X)(v_q) =
\vlft_{v_q}(X(q))$. The vertical lift of a distribution $\D$ on $\Q$
is the distribution $\vlft(\D)$ on $T\Q$ defined by $\vlft(\D)(v_q) =
\vlft_{v_q}(\D(q))$. 

The control-affine system on $T\Q$ associated with the Lagrangian control
system~\eqref{eq:Lagrangian_control_system} is given by
\begin{equation}\label{eq:control_affine}
\dot X = S(X) - \vlft(\grad P)(X) + \sum_{i=1}^m \tau_i \vlft\big(
(F^i)^\sharp \big) (X).
\end{equation}
In the above, the vector field $S \in \X(TQ)$ is the {\em geodesic
  spray} associated with the metric $g$, and it has the property that
the integral curves of $S$ project to geodesics of the metric $g$ on
$\Q$ via the canonical projection $\pi : T\Q \to \Q$.  The control
system~\eqref{eq:control_affine} is equivalent
to~\eqref{eq:Lagrangian_control_system} in the following sense.  All
maximal base integral curves of~\eqref{eq:Lagrangian_control_system}
are projections under $\pi: T\Q \to \Q$ of maximal integral curves
of~\eqref{eq:control_affine}. Vice versa, the projection of any
maximal integral curve of~\eqref{eq:control_affine} is a maximal base
integral curve of~\eqref{eq:Lagrangian_control_system}.

In coordinates, equations~\eqref{eq:Lagrangian_system}
and~\eqref{eq:Lagrangian_control_system} take on the familiar form of
the Euler-Lagrange equations (e.g.,~\cite[Chapter
  3]{ArnoldClassicalMech}). More precisely, let $q : I \to \Q$ be a
base integral curve of a Lagrangian system $(\Q,g,P)$. Given a chart
$(U,\phi)$ for $\Q$, let $x = \phi \circ q$ be the coordinate
representation of $q$, and let $\hat L(x,\dot
x):=L(\phi^{-1}(x),(d\phi^{-1})_x \dot x)$, where $L$ is the
Lagrangian function defined in~\eqref{eq:Lagrangian}. Then,
$x=(x^1,\ldots,x^n): I \to \Re^n$ satisfies the Euler-Lagrange
equations
\[
\frac{d}{dt} \frac{\partial \hat L}{\partial \dot x^i} -
\frac{\partial \hat L}{\partial x^i} = 0, \ i =1,\ldots, n.
\]
Vice versa, if $x: I \to \Re^n$ satisfies the above Euler-Lagrange
equations, then $q: I \to \Q$, $q:=\phi^{-1} \circ x$ is a base
integral curve of $(\Q,g,P)$. An analogous property holds for
Lagrangian control systems $(\Q,g,P,\sF)$, where now $q: I \to \Q$ is
a base integral curve of $(\Q,g,P,\{F^i\}_{i=1,\ldots,m})$ if, and
only if, $x=\phi \circ q$ is a solution of the forced Euler-Lagrange
equations
\begin{equation}\label{eq:EulerLagrange_forced}
\frac{d}{dt} \frac{\partial \hat L}{\partial \dot x^j} -
\frac{\partial \hat L}{\partial x^j} = \sum_{i} B_{ij} (x) \tau_i, \ j
=1,\ldots, n.
\end{equation}
In the above $B_{ij}(x) := F^j(\der{i}{\phi^{-1}(x)})$ is the $i$-th
coefficient of the expansion of $F^j$ in the local frame
$\{\de{1},\ldots,$ $\de{n}\}$ induced by the chart.  Let $D(x)$ be the
matrix with components $D_{ij}(x) = g_{ij}(\phi^{-1}(x))$, and set
$C(x,\dot x)=D(x) \tilde C(x,\dot x)$, where $\tilde C_{kj}(x,\dot x)
= \sum_{i} \Gamma_{ij}^k ( \phi^{-1}(x)) \dot x_i$.  Let $\P:=P
\circ \phi^{-1}$, and let $B(x)$ be the matrix with elements
$B_{ij}(x)$. Then the Euler-Lagrange
equations~\eqref{eq:EulerLagrange_forced} take on the familiar form
\begin{equation}\label{eq:mechanical_control_system}
D(x) \ddot x + C(x,\dot x)\dot x+\nabla_x \P(x) = B(x) \tau,
\end{equation}
where $\tau$ is the vector whose components are the control inputs
$\tau_i$ in~\eqref{eq:EulerLagrange_forced}.

\section{Review of Virtual Holonomic Constraints }\label{sec:background_vhc}

The configuration manifold $\Q$ of a robot whose joints are revolute
or prismatic is a generalized cylinder. In other words, an element of
$\Q$ may be represented as an $n$-tuple $(q_1,\ldots,q_n)$, where each
$q_i$ is either in $\Re$ if the $i$-th joint is prismatic, or in
$\Se^1$ if the $i$-th joint is revolute. In this case, the Lagrangian
control system~\eqref{eq:Lagrangian_control_system} admits a global
coordinate representation of the form~\eqref{eq:mechanical_control_system}:
\begin{equation}\label{eq:EulerLagrange_forced2}
D(q) \ddot q + C(q,\dot q) \dot q + \nabla_q \P(q) = B(q) \tau.
\end{equation}
In this section, we review basic facts concerning VHCs for the class
of mechanical control systems~\eqref{eq:EulerLagrange_forced2}. The
rest of this paper will be devoted to the generalization of these
results to the coordinate-free setting. We assume, throughout, that
the matrix $B$ has full-rank $m$.
\begin{definition}[Virtual holonomic constraint in coordinates,~\cite{6286994}] 
\label{defn:VHC:coordinates}
A {\em virtual holonomic constraint of order $k$} for
system~\eqref{eq:EulerLagrange_forced2} is a relation $h(q)=0$, where
$h:\Q \to \Re^k$ is a smooth function such that $0$ is a regular value
of $h$ and the set
\begin{equation}\label{eq:constraint_manifold_coordinates}
\Gamma = \{(q,\dot q) \in T\Q : h(q)=0, \ dh_q \dot q =0\}
\end{equation}
is controlled invariant for~\eqref{eq:EulerLagrange_forced2}. The set
$\Gamma$ is called the {\em constraint manifold} associated with the
VHC $h(q)=0$.
\end{definition}
Requiring $\Gamma$ to be controlled invariant means requiring that
there exists a smooth feedback $\tau = \tau^\star(q,\dot q)$ rendering
$\Gamma$ positively invariant for the closed-loop system (see,
e.g.,~\cite[Definition 11.1]{NijSch90}). If we let $\C :=h^{-1}(0)$,
then the hypothesis that $\rank(dh_q)=k$ for all $q\in h^{-1}(0)$
implies that $\C$ is a closed embedded submanifold of $\Q$. Moreover,
the constraint manifold $\Gamma$
in~\eqref{eq:constraint_manifold_coordinates} is the tangent bundle of
$\C$, $\Gamma = T\C$.

In the context of nonlinear control, the constraint manifold
associated with a VHC $h(q)=0$ is the zero dynamics manifold of
system~\eqref{eq:EulerLagrange_forced2} with output $e= h(q)$
(see~\cite{Isi95}). A special case of interest is when this output
function yields a well-defined relative degree, as in the next
definition.
\begin{definition}[Regular VHC in coordinates,~\cite{6286994}]
\label{defn:regularVHC:coordinates}
A relation $h(q)=0$ is a {\em regular VHC of order $k$} for
system~\eqref{eq:EulerLagrange_forced2} if $h: \Q \to \Re^k$ is
smooth, and system~\eqref{eq:EulerLagrange_forced2} with output
$e=h(q)$ has vector relative degree\footnote{This means~\cite{Isi95}
  that the control input $\tau$ appears nonsingularly in the second
  time derivative, $\ddot e$, of $e$ along solutions
  of~\eqref{eq:EulerLagrange_forced2}.}  $\{2,\ldots,2\}$ everywhere
on the constraint manifold $\Gamma$
in~\eqref{eq:constraint_manifold_coordinates}, i.e., $\rank ( dh_q
D^{-1}(q)B(q) ) =k$ for all $ q \in h^{-1}(0)$.
\end{definition}

Since the matrix $dh_q D^{-1}(q)B(q)$ has dimension $k \times m$, for
a regular VHC it must hold that the number of constraints, $k$, be
less than or equal to the number of controls, $m$.  The next result
provides a geometric interpretation of the regularity condition.
\begin{prop}[\cite{JanMagCon12}]
A relation $h: \Q \to \Re^k$ is a regular VHC for
system~\eqref{eq:EulerLagrange_forced2} if, and only if, letting $\C =
h^{-1}(0)$,
\begin{equation}\label{eq:transversality:coordinates}
(\forall q \in \C) \ T_q \C + \image(D^{-1}(q) B(q)) = T_q \Q.
\end{equation}
\end{prop}
In light of the proposition above, the VHC $h(q)=0$ is regular if, for
each $q \in \C$, the mechanical system can produce control
accelerations $D^{-1}B \tau$ in any direction transversal to $T_q
\C$. In the special case $k=m$, when the number of constrains is equal
to the number of controls, the subspace sum
in~\eqref{eq:transversality:coordinates} becomes direct.

Regular VHCs are important in two respects. First, since the output
$e=h(q)$ yields a well-defined vector relative degree on the
constraint manifold $\Gamma$, one may use input-output linearization
to asymptotically stabilize $\Gamma$. For this, some technical
assumptions on $h$ and its differential are required,
see~\cite{6286994}.  Second, when $k=m$, there is a unique smooth
feedback $\tau^\star: \Gamma \to \Re^m$ rendering $\Gamma$ invariant,
resulting in constrained dynamics on $\Gamma$ described by an
autonomous differential equation. This is stated in the next
proposition for the case\footnote{The more general case $k \leq
  m \leq n$ is addressed in~\cite{JanMagCon12}.}  $m=n-1$. In what follows,
%% we assume
%% that the mechanical system~\eqref{eq:EulerLagrange_forced2} has degree
%% of underactuation one, i.e., $m=n-1$, and
we let $\bperp:\Q \to \Re^{1 \times n} \backslash \{0\}$ be a smooth
left annihilator of $B$ of rank one everywhere on $h^{-1}(0)$.
\begin{prop}[\cite{6286994}]\label{prop:reduced_dynamics:coordinates}
Let $m=n-1$, and let $h(q)=0$ be a regular VHC of order $n-1$ for
system~\eqref{eq:EulerLagrange_forced2}. Then
  there exists a unique smooth feedback $\tau^\star: \Gamma \to \Re^m$
  rendering $\Gamma$ in~\eqref{eq:constraint_manifold_coordinates}
  invariant, and the resulting closed-loop dynamics on $\Gamma$ are
  given as follows.  Let $\varphi:\Theta \to \Q$ be a regular
  parametrization of the curve $\C=h^{-1}(0)$, where $\Theta = \Se^1$
  if $\C$ is a Jordan curve and $\Theta = \Re$ otherwise, and let
  $(q,\dot q)= (\varphi(\th), \varphi'(\th) \dot \th)$. The
  closed-loop dynamics on $\Gamma$ are given by
\begin{equation}\label{eq:constrained_dynamics}
\ddot \th = \Psi_1(\th) + \Psi_2(\th) \dot{\th}^2,
\end{equation}
where $(\th,\dot \th) \in \Theta \times \Re$ and
\begin{equation}\label{eq:Psi_functions}
\begin{aligned}
& \Psi_1(\th) = -\frac{\bperp\nabla_q \P}{\bperp D
    \varphi'}\bigg|_{q=\varphi(\th)}, \\
& \Psi_2(\th) = -\frac{\bperp D \varphi'' + \sum_{i=1}^n (\bperp D)_i
    \varphi'{}\trans \Gamma^i \varphi'}{\bperp D
    \varphi'}\Bigg|_{q=\varphi(\th)},
\end{aligned}
\end{equation}
where $(\Gamma^i)_{jk} =\Gamma^i_{jk}$ is the matrix containing the
Christoffel symbols of the metric $g_q(v_q,w_q) = v_q\trans D(q) w_q$.
\end{prop}
The dynamics in~\eqref{eq:constrained_dynamics} are called the {\em
  constrained (or reduced) dynamics} associated with the VHC
$h(q)=0$. The next result, taken from~\cite{MohMagCon17},
characterizes when the constrained
dynamics~\eqref{eq:constrained_dynamics} possess a Lagrangian
structure.
\begin{thm}[\cite{MohMagCon17}]\label{thm:ILP:1dim}
Define a map $\pi : \Re \to \Theta$ as
\[
\pi(x) = \begin{cases} x & \text{ if } \ \Theta = \Re \\ x \Mod 2 \pi &
  \text{ if } \ \Theta = \Se^1.
\end{cases}
\]
Define smooth functions $\hMC,\hPC: \Re \to \Re$ as,
\[
\begin{aligned} 
\hMC(x)&:=\exp \left( -2 \int_0^x \Psi_2 \circ \pi(z) d z \right),
\\
\hPC(x)&:= -\int_0^x \big( \Psi_1 \circ \pi(z) \big) \hMC(z) d z,
\end{aligned}
\]
and (generally multi-valued) functions $\MC,\PC : \Theta \rightrightarrows \Re$ as
\[
\MC:=\hMC \circ \pi^{-1},  \ \PC:=\hPC \circ \pi^{-1}.
\]
Let
\begin{equation}\label{eq:ILP:Lagrangian}
L(\th,\dot \th) = \frac 1 2  \MC(\th) \dot
\th^2 - \PC (\th).
\end{equation}
Then the following statements are true.
\begin{enumerate}[(a)]
\item If $\Theta = \Re$, then~\eqref{eq:constrained_dynamics} is a
  Lagrangian system with Lagrangian $L$ in~\eqref{eq:ILP:Lagrangian}.

\item If $\Theta = \Se^1$, then~\eqref{eq:constrained_dynamics} is
  Lagrangian if, and only if, $\hMC$ and $\hPC$ are $2 \pi$-periodic,
  in which case $\MC, \PC$ in~\eqref{eq:ILP:Lagrangian} are
  single-valued and smooth, and $L$ in~\eqref{eq:ILP:Lagrangian} is
  the Lagrangian function of~\eqref{eq:constrained_dynamics}.
\end{enumerate}
\end{thm}
As will become apparent in the development that follows, the results
reviewed in this section are a special case of the general theory
developed in this paper.

\section{Coordinate-free formulation of virtual holonomic constraints}\label{sec:coordinate_free}
In this section we reformulate and generalize the theory of
Section~\ref{sec:background_vhc} in a coordinate-free context.  We
consider throughout a Lagrangian control system
$(\Q,g,P,\sF)$ with equations of motion
\begin{equation}\label{eq:Lagrangian_control_system2}
\nabla_{\dot q} \dot q = -\grad P(q) + \sum_{i=1}^m (F^i)^\sharp_{q}
\tau_i.
\end{equation}
We assume that the one-forms $\sF=\{F^1,\ldots,F^m\}$ are independent, and
define
%% the {\em force codistribution}
%% %
%% %
%% %
%% \begin{equation}\label{eq:force_distribution}
%% \Omega_F = \spn\{F^1,\ldots,F^m\},
%% \end{equation}
%% %
%% %
%% %
%% and
the {\em acceleration distribution}
\begin{equation}\label{eq:acceleration_distribution}
\D_A = \spn\{(F^1)^\sharp,\ldots, (F^m)^\sharp\}.
\end{equation}
We recall that base integral curves
of~\eqref{eq:Lagrangian_control_system2} are projections onto $\Q$ of
integral curves of the control affine system
\begin{equation}\label{eq:control_affine2}
\dot X = S(X) - \vlft(\grad P)(X) + \sum_{i=1}^m \tau_i \vlft\big(
(F^i)^\sharp \big) (X)
\end{equation}
via the canonical projection map $\pi : T \Q \to \Q$. Note that
\[
\vlft(\D_A) = \spn\big\{ \vlft( (F^1)^\sharp),\ldots, \vlft(
(F^m)^\sharp\big)\}.
\]

\subsection{VHC definitions and relationships}

%We begin with a definition of controlled invariant submanifolds for
%system~\eqref{eq:Lagrangian_control_system2}.
%
%
%
\begin{definition}[Controlled invariant submanifold]\label{defn:controlled_invariance}
Let $\C$ be a closed embedded submanifold of $\Q$. The tangent bundle
$T\C$ is {\em controlled invariant}
for~\eqref{eq:Lagrangian_control_system2} if there exists a smooth
feedback $\tau^\star = (\tau_1^\star,\dots,\tau_m^\star): T\C\to
\Re^m$ such that for each $q_0 \in \C$ and each $v_{q_0} \in T_{q_0}
\C$, the maximal base integral curve $q : I \to \Q$
of~\eqref{eq:Lagrangian_control_system} with feedback $\tau =
\tau^\star(q,\dot q)$ and initial condition $(q_0,v_{q_0})$ satisfies
$q(t) \in \C$ for all $t \in I$.
\end{definition}
In reference to the control-affine system~\eqref{eq:control_affine2},
the definition above can be rephrased as the requirement that there
exists a smooth feedback rendering $T \C$ an invariant set for the
closed-loop system, which is the standard concept of controlled
invariance of submanifolds used in control theory (see,
e.g.,~\cite{Isi95}).

\begin{definition}[Virtual holonomic constraint] \label{defn:vhc}
A {\em virtual holonomic constraint of order $k$} for the Lagrangian
control system~\eqref{eq:Lagrangian_control_system} is a closed
embedded submanifold $\C$ of $\Q$ of codimension $k$ such that $T\C$
is controlled invariant for~\eqref{eq:Lagrangian_control_system}. The
set $T \C$ is called the {\em constraint manifold}.
\end{definition}
\begin{definition}[Regular VHC of order $m$] \label{defn:regularVHC}
A closed embedded submanifold $\C$ of $\Q$ is a {\em regular VHC of
  order $m$} for the Lagrangian control
system~\eqref{eq:Lagrangian_control_system} if $\C$ has codimension
$m$ and
\begin{equation}\label{eq:transversality_condition}
(\forall q \in \C) \ T_q \C \oplus \D_A(q) = T_q \Q,
\end{equation}
where $\D_A$ is the acceleration distribution defined
in~\eqref{eq:acceleration_distribution}.
\end{definition}
\begin{figure}[htb]
\psfrag{T}{$T_q \Q$}
\psfrag{U}{$T_q \C$}
\psfrag{D}{$\D_A(q)$}
\psfrag{q}{$q$}
\psfrag{C}{$\C$}
\psfrag{Q}{$\Q$}
\centerline{\includegraphics[width=.6\textwidth]{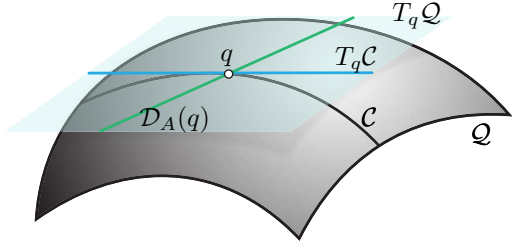}}
\caption{Transversality condition in the definition of regular VHC.}
\label{fig:tranversality}
\end{figure}
The transversality condition~\eqref{eq:transversality_condition},
illustrated in Figure~\ref{fig:tranversality},
generalizes~\eqref{eq:transversality:coordinates} in the case when the
number of constraints, $k$, is equal to the number of controls,
$m$. 

A regular VHC is a VHC in the sense of Definition~\ref{defn:vhc}, as
the next result shows.
\begin{prop}\label{prop:regVHC_is_VHC} 
If a closed embedded submanifold $\C$ of $\Q$ is a regular VHC of
order $m$ for system~\eqref{eq:Lagrangian_control_system2}, then $\C$
is also a VHC in the sense of Definition~\ref{defn:vhc}, and the
smooth feedback $\tau^\star : T\C \to \Re^m$ rendering $T\C$ invariant
is unique.
\end{prop}
We need the following lemma, whose proof is in the appendix.

\begin{lem}\label{lem:controlled_invariance}
Consider the Lagrangian control
system~\eqref{eq:Lagrangian_control_system2} and its associated
control-affine system~\eqref{eq:control_affine2}. If $\C$ is a regular
VHC, then for each $X_q \in T\C$,
\begin{equation}\label{eq:controlled_invariance}
S(X_q) - \vlft(\grad P)(X_q) \in T_{X_q} T\C \oplus \vlft(\D_A)(X_q).
\end{equation}
\end{lem}

\begin{proof}[Proof of Proposition~\ref{prop:regVHC_is_VHC}] 
Suppose $\C$ is a regular VHC for the Lagrangian control
system~\eqref{eq:Lagrangian_control_system2}. We need to show that the
closed embedded submanifold $T\C \subset T\Q$ is controlled invariant
for the control-affine system~\eqref{eq:control_affine2}, and the
smooth feedback rendering it invariant is unique. By
Lemma~\ref{lem:controlled_invariance}, for each $X_q \in T \C$ we have
that
\begin{equation}\label{eq:coninv}
S(X_q) - \vlft(\grad P)(X_q) \in T_{X_q} T\C
  \oplus \spn\{\vlft\big( (F^i)^\sharp \big), i\in \{1,\ldots,m\}\},
\end{equation}
from which it follows that there is a unique vector $\tau^\star(X_q) =
(\tau_1^\star(X_q),\ldots,\tau_m^\star(X_q))$ such that
\begin{equation}\label{eq:tangency_vhc}
S(X_q) - \vlft(\grad P)(X_q) + \sum_{i=1}^m \tau_i^\star(X_q)
\vlft\big( (F^i)^\sharp \big) \in T_{X_q} T\C.
\end{equation}
The map $T\C \to \Re^m$, $X_q \to \tau^\star(X_q)$ is smooth because
in any set of local coordinates $\hat X$ on $T\C$, the
requirement~\eqref{eq:tangency_vhc} can be expressed as a matrix
equation of the form $ A(\hat X) \tau = b(\hat X)$, where,
by~\eqref{eq:coninv}, $\rank A = m$ and $b(\hat X) \in \image A(\hat
X)$. The unique solution $\tau(\hat X)$ of this equation is $\tau(\hat
X) = \big( A(\hat X)\trans A(\hat X) \big)^{-1} A(\hat X)\trans b(\hat
X)$, which is a smooth function.
In conclusion, there exists a unique smooth feedback $\tau^\star: T\C
\to \Re^m$ such that the closed-loop vector field given
by~\eqref{eq:control_affine} with $\tau_i = \tau_i^\star(X)$ is
tangent to $T \C$. By~\cite[Theorem 2.1]{Clarke},
$T\C$ is an invariant set for the closed-loop vector field, and thus
$\C$ is a VHC in the sense of Definition 4.2.
\end{proof}

Definition~\ref{defn:regularVHC} of regular VHCs is a coordinate-free
generalization of Definition~\ref{defn:regularVHC:coordinates} in the
following sense. When $\C$ is globally described by the zero level set
of a smooth submersion $h: \Q \to \Re^m$, then
$T\C = \{v_q \in T\Q: h(q)=0,\ dh_q v_q=0\}$, and
system~\eqref{eq:Lagrangian_control_system2} with output function $e =
h(q)$ has vector relative degree $\{2,\ldots,2\}$. Indeed,
differentiating each component of the output, $e_i = h_i(q)$, twice along the
base integral curves of~\eqref{eq:Lagrangian_control_system2} one can
show that
%% \[
%% \dot e_i = (dh_i)_{q(t)} \dot q = g(\grad h_i, \dot q),
%% \]
%% and by the compatibility property~\eqref{eq:compatibility} of
%% Riemannian connections, we have
%% \[
%% \ddot e_i = g(\nabla_{\dot q} \grad h_i,\dot q) + g (\grad h_i,
%% \nabla_{\dot q} \dot q).
%% \]
%% Substituting $\nabla_{\dot q} \dot q$
%% from~\eqref{eq:Lagrangian_control_system2}, we obtain
%
%
%
\begin{equation}\label{eq:eddot}
\ddot e_i = g(\nabla_{\dot q} \grad h_i, \dot q) - g(\grad h_i, \grad
P) + \sum_{j=1}^m b_{ij}(q) \tau_j, \ i =1,\ldots, k,
\end{equation}
where $b_{ij} = g(\grad h_i,(F^j)^\sharp)$. The transversality
condition~\eqref{eq:transversality_condition} in
Definition~\ref{defn:regularVHC} implies that the $m \times m$ matrix
with component $b_{ij}$ is invertible on $\C$, and thus
system~\eqref{eq:Lagrangian_control_system2} with output function $e =
h(q)$ has vector relative degree $\{2,\ldots,2\}$.

Just like in Section~\ref{sec:background_vhc}, the
expression~\eqref{eq:eddot} suggests a way to asymptotically
stabilize\footnote{Provided that certain technical assumptions hold
  for $h$ and $dh$ (see~\cite{6286994}).} the constraint manifold
$T\C$ using an input-output linearizing feedback 
\[
\tau_i^\star = \sum_j b^{ij}(q) \left[ -g(\nabla_{\dot q} \grad
  h_j,\dot q) + g(\grad h_j, \grad P) - K_{p,j} h_j - K_{d,j} (dh_j)_q
  \dot q\right],
\]
where $b^{ij}$ is the $(i,j)$-th element of the inverse of the matrix
$(b_{ij})$, and $K_{p,j}$, $K_{d,j}$ are positive design parameters.

The restriction of $\tau^\star$ above 
to $T\C$ is the unique feedback rendering $T\C$ invariant predicted
by Proposition~\ref{prop:regVHC_is_VHC}, and it is given by
\[
\tau_i^\star |_{T\C} = \sum_j b^{ij}(q) \left[ -g(\nabla_{\dot q} \grad
  h_j,\dot q) + g(\grad h_j, \grad P) \right].
\]
For base integral curves $q(t)$ in $\C$ it holds that $\dot q \in T_q
\C$. Moreover, since $\C = h^{-1}(0)$, $\grad h_i(q) \in T_q
\C^\perp$.  The Weingarten equation~\cite{book:598491} then gives
$g(\nabla_{\dot q} \grad h_j,\dot q) = -g(\grad h_j,\mathrm{I\!I}(\dot
q,\dot q))$, where $\mathrm{I\!I}$ is the second fundamental form of
$\C$. Thus the unique feedback rendering $T\C$ invariant is
\[
\tau_i^\star|_{T\C} = \sum_j b^{ij}(q) g(\grad h_j, \mathrm{I\!I}(\dot
q,\dot q) + \grad P).
\]

\subsection{Constrained dynamics}\label{sec:constrained_dynamics}
      
Our next objective is to characterize the constrained dynamics on
$T\C$, by which we mean the closed-loop dynamics resulting from the
application of the unique smooth feedback $\tau^\star: T \C \to \Re^m$
rendering $T \C$ invariant.  We would like a coordinate-free
generalization of Proposition~\ref{prop:reduced_dynamics:coordinates}
valid for any $m$, not just $m=n-1$.  As we now show, such dynamics
are described by a special affine connection on $\C$ induced by the
VHC. This so-called {\em induced connection} was originally developed
in the context of affine differential geometry (see~\cite[Chapter
  2]{nomizu1994affine}).  We adopt it in the context of regular VHCs.
 
Before giving a formal definition of the induced connection, we
present the basic idea behind it. If $\C$ is a regular VHC, the
transversality condition~\eqref{eq:transversality_condition} in
Definition~\ref{defn:regularVHC} states that, for each $q \in \C$, the
tangent space $T_q \Q$ is the direct sum of $T_q \C$ and $\D_A(q)$. We
may then define the projection $\sigma_q: T_q \Q \to T_q \C$ of the
vector space $T_q \Q$ onto $T_q \C$ along the subspace $\D_A(q)$.  The
map $\sigma_q$ is uniquely determined by the following properties:
\begin{enumerate}[(i)]
\item $\sigma_q^2 = \sigma_q$,
\item $\image \sigma_q = T_q \C$,
\item $\Ker \sigma_q = \D_A(q)$.
\end{enumerate}
Now consider the vector bundle map $\sigma: T\Q|_\C \to T\C, \ w_q
\mapsto \sigma_q(w_q)$, illustrated in Figure~\ref{fig:projection}.
\begin{figure}[htb]
\psfrag{T}{$T_q \Q$}
\psfrag{U}{$T_q \C$}
\psfrag{D}{$\D_A(q)$}
\psfrag{q}{$q$}
\psfrag{C}{$\C$}
\psfrag{Q}{$\Q$}
\psfrag{w}{$w_q$}
\psfrag{s}{$\sigma_q(w_q)$}
\centerline{\includegraphics[width=.7\textwidth]{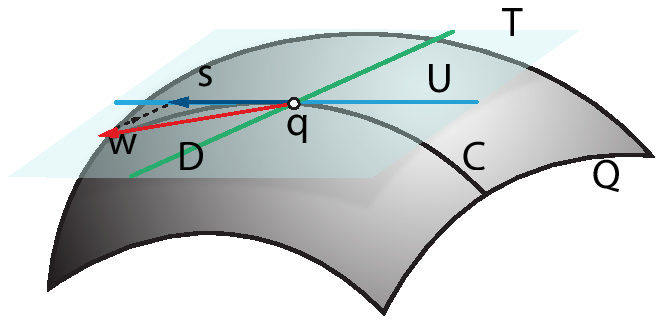}}
\caption{The vector bundle map $\sigma: T \Q|_\C \to T \C$.}
\label{fig:projection}
\end{figure}

Since the acceleration distribution $\D_A(q)$ is smooth, so is
$\sigma$. Using $\sigma$, we define a new connection on $\C$ as
follows. Given two vector fields $X,Y \in \X(\C)$, $\nabla_X Y$ is
generally not a vector field on $\C$, but its projection
$\sigma(\nabla_X Y)$ is, and the next theorem shows that this
operation identifies an affine connection on $\C$.

Before presenting the theorem, we need to justify the notation
$\nabla_X Y$ for vector fields $X,Y \in \X(\C)$, since the affine
connection $\nabla$ accepts vector fields on $\Q$.  Consider arbitrary
smooth extensions\footnote{These exist by~\cite[Problem
    8.15]{lee2012introduction}.}  $\tilde X, \tilde Y$ of $X,Y$ on a
neighbourhood of $\C$ in $\Q$ such that $\tilde X|_\C = X$ and $\tilde
Y|_\C=Y$.  Given any $p \in \C$, by~\cite[Exercise 4.7,
  p.58]{book:598491}, the value of $\nabla_{\tilde X} \tilde Y(p)$
depends only on $\tilde X_p$ (and thus $X_p$) and the value of $\tilde
Y$ along any smooth curve $\gamma : (-\varepsilon,\varepsilon) \to \Q$
such that $\gamma(0)=p$ and $\dot \gamma(0) = X_p$. Since $X_p \in T_p
\C$, we may pick a curve $\gamma$ contained in $\C$, so that the value
of $\tilde Y$ along $\gamma$ coincides with that of $Y$. Therefore, on
$\C$ the function $\nabla_{\tilde X} \tilde Y$ is uniquely determined
by $X,Y$.  These considerations justify the slight abuse of notation
$\nabla_X Y$ for vector fields $X,Y \in \X(\C)$.
\begin{thm}[\cite{nomizu1994affine}]\label{thm:induced_connection}
Let $\C$ be a regular VHC of order $m$ for the Lagrangian control
system~\eqref{eq:Lagrangian_control_system2}, and define the map
$\nablaC: \X(\C) \times \X(\C) \to \X(\C)$ as
\begin{equation}\label{eq:induced_connection}
%\nablaC_X Y := \sigma \circ (\iota^* \nabla)_X Y,
\nablaC_X Y :=\sigma \big( \nabla_{X} Y \big),
\end{equation}
where $\nabla$ is the Riemannian connection of $(\Q,g)$. The map
$\nablaC$ is a symmetric affine connection on $\C$.
\end{thm}
The above result is mentioned in~\cite[Chapter 2,
  p. 28]{nomizu1994affine}.  The straightforward proof is omitted.

We call $\nablaC$ the {\em induced connection}, or the connection
induced by the regular VHC $\C$. While $\nablaC$ is symmetric, it is
generally not a Riemannian connection with respect to the induced
Riemannian metric on $\C$. This fact is discussed in the next section.
Now the main result of this section. In what follows, let $\iota : \C
\to \Q$ denote the inclusion map.
\begin{thm}\label{thm:reduced_dynamics}
If $\C$ is a regular VHC of order $m$ for the Lagrangian control
system~\eqref{eq:Lagrangian_control_system2}, then the constrained
dynamics on $T\C$ are described by the equation of motion
\begin{equation}\label{eq:constrained_dynamics:connection}
\nablaC_{\dot q} \dot q = -\sigma_q (\grad P(q)),
\end{equation}
in the following sense. If $q: I \to \Q$ is a maximal base integral
curve of~\eqref{eq:Lagrangian_control_system2} such that $q(I) \subset
\C$, then $q: I \to \C$ is a maximal base integral curve of
system~\eqref{eq:constrained_dynamics:connection}.  Vice versa, if $q:
I \to \C$ is a maximal base integral curve
of~\eqref{eq:constrained_dynamics:connection}, then $\iota \circ q$ is
a maximal base integral curve
of~\eqref{eq:Lagrangian_control_system2}.
\end{thm}
\begin{proof}
Let $q: I \to \Q$ be a base integral curve
of~\eqref{eq:Lagrangian_control_system2} such that $q(I) \subset
\C$. Then
\[
\nabla_{\dot q} \dot q +\grad P(q) - \sum_{i=1}^m (F^i)^\sharp
\tau_i^\star(q) =0,
\]
where $\tau_i^\star(q)$ is the $i$-th component of the unique feedback
$\tau^\star: T \C \to \Re^m$ rendering $T\C$ invariant (see
Proposition~\ref{prop:regVHC_is_VHC}).  Using $\sigma_q$ to project
both sides of the above identity onto $T_q \C$ and using the fact that
$\sigma_q((F^i)^\sharp)=0$, we get
\[
\sigma_q (\nabla_{\dot q} \dot q +\grad P(q))=0.
\]
By Theorem~\ref{thm:induced_connection} we get
\[
\nablaC_{\dot q} \dot q + \sigma_q (\grad P(q)) =0,
\]
which proves that $q : I \to \C$ is a base integral curve
of~\eqref{eq:constrained_dynamics:connection}.

Now let $q :I \to \C$ be a base integral curve
of~\eqref{eq:constrained_dynamics:connection}. Then,
\[
\nablaC_{\dot q} \dot q + \sigma_q (\grad P(q)) =0,
\]
or
\[
\sigma_q (\nabla_{\dot q} \dot q + \grad P(q) ) =0.
\]
Since $\Ker \sigma_q = \D_A (q)$, we have
\[
\nabla_{\dot q} \dot q + \grad P(q) \in \D_A(q),
\]
from which it follows that, for each $t \in I$, there exists $\bar
\tau(q(t))=(\bar \tau_1(q(t)),\ldots,\bar \tau_m(q(t))) \in \Re^m$
such that
\[
\nabla_{\dot q(t)} \dot q(t) + \grad P(q(t)) = \sum_{i=1}^m
(F^i)^\sharp(q(t)) \tau_i(q(t)).
\]
By the uniqueness of the feedback $\tau^\star$ rendering $T \C$
invariant (see Proposition~\ref{prop:regVHC_is_VHC}), it must hold
that $\bar \tau = \tau^\star$, a smooth feedback. This proves that
$\iota(q)$ is an integral curve
of~\eqref{eq:Lagrangian_control_system2}.

We now prove maximality. Suppose, by way of contradiction, that $q : I
\to \Q$ a maximal base integral curve
of~\eqref{eq:Lagrangian_control_system2} such that $q(I) \subset \C$
but the corresponding base integral curve $q : I \to \C$
of~\eqref{eq:constrained_dynamics:connection} is not maximal. Let
$\tilde q : \tilde I \to \C$, $\tilde I \supset I$, be the unique
maximal base integral curve
of~\eqref{eq:constrained_dynamics:connection} such that $\tilde q|_I =
q$. Then $\iota (\tilde q)$ is a base integral curve
of~\eqref{eq:Lagrangian_control_system2} with a larger interval of
existence than $q$, which contradicts the maximality of $q$. In an
analogous way one shows that if $q : I \to \C$ is a maximal base
integral curve of~\eqref{eq:constrained_dynamics:connection} then
$\iota(q)$ is maximal for~\eqref{eq:Lagrangian_control_system2}.
\end{proof}

\subsection{Constrained dynamics in coordinates}\label{sec:constrained_dynamics:coordinates}

We now characterize the constrained dynamics in coordinates.  Pick a
coordinate chart for $\Q$, $(U,\phi)$, with $\phi: U \to \hat U
\subset \Re^n$ and $\C \cap U \neq \emptyset$.  Letting $x = \phi(q) =
(x^1(q),\ldots,x^n(q))$, the equations of
motion~\eqref{eq:Lagrangian_control_system2} in $x$ coordinates read
as (cf.~\eqref{eq:mechanical_control_system}),
\[
D(x) \ddot x + C(x,\dot x) \dot x + \nabla_x \P(x) = B(x) \tau.
\]
{\bf Tangent space of $\C$.} The chart domain $U$ can always be chosen
small enough  that the local representation of the constraint
manifold, $\hat \C = \phi(\C \cap U)$, is the image of a
diffeomorphism $\varphi: W \subset \Re^{n-m} \to \hat \C$,
$\th=(\th^1,\ldots,\th^{n-m}) \mapsto \varphi(\th)$.  Using this
parametrization, we have $T_{\varphi(\th)} \hat \C = \image(d
\varphi_\th)$. Thus, letting
\[
V^i(x) :=d \varphi_{\varphi^{-1}(x)} (\partial / \partial \th^i) =
\partial_{\th^i} \varphi(\varphi^{-1}(x)), \ i=1,\ldots, n-m,
\]
we have
\[
T \hat \C = \spn\{V^1,\ldots, V^{n-m}\}.
\]
This construction is depicted in Figure~\ref{fig:constrained_dynamics:coordinates}.
\begin{figure}[htb]
\psfrag{s}{$s$}
\psfrag{d}{$\partial/\partial s^i$}
\psfrag{v}{$V^i(x)$}
\psfrag{x}{$x$}
\psfrag{C}{$\C$}
\psfrag{Q}{$\Q$}
\psfrag{U}{$U$}
\psfrag{V}{$\hat U \subset \Re^n$}
\psfrag{D}{$\hat \C$}
\psfrag{p}{$\varphi$}
\psfrag{P}{$d \varphi_s$}
\psfrag{W}{$W$}
\psfrag{R}{$\Re^{n-m}$}
\psfrag{T}{$\phi$}
\centerline{\includegraphics[width=.7\textwidth]{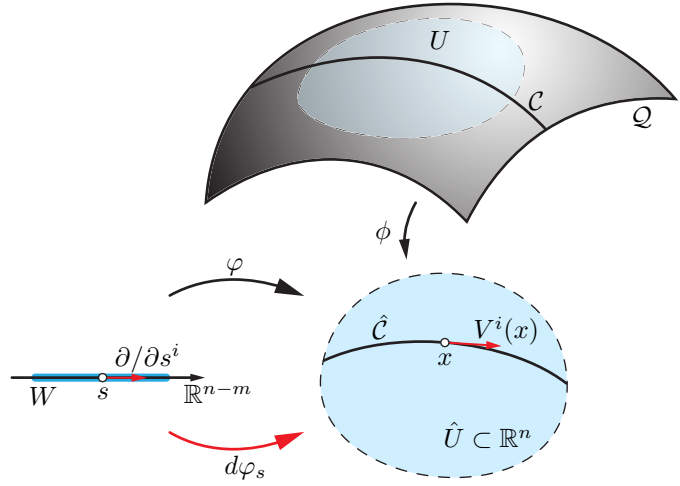}}
\caption{Coordinate systems used in
  Section~\ref{sec:constrained_dynamics:coordinates}.}
\label{fig:constrained_dynamics:coordinates}
\end{figure}

\medskip
\noindent
{\bf Projection map.}  In coordinates, we have
\[
d\phi_{\phi^{-1}(x)} \big( \D_A(\phi^{-1}(x)) \big) = \image
D^{-1}(x)B(x).
\]
Letting $\bperp$ be a full-rank left-annihilator of $B$, the
coordinate representation of the projection map $\sigma$ is the map
$\hat \sigma : T \hat U \to T\hat \C$ defined as
\begin{equation}
\label{eq:sigma:coordinates}
\hat \sigma_{x} (v_{x}) = d \varphi_\th \big( (\bperp D
d\varphi_\th)^{-1} \bperp D \big)\big|_{x =\varphi(\th)} ( v_{x} ).
\end{equation}
Indeed, one can readily verify that $\hat \sigma_x^2 = \hat \sigma_x$,
$\image(\sigma_x) = T_x \hat \C$, and $\Ker \sigma_x =
\image(D^{-1}(x) B(x))$. These properties imply that $\hat \sigma_x$
is the projection onto $T_x \hat \C$ along the subspace
$\image(D^{-1}(x)B(x))$, as required.

\medskip
\noindent
{\bf Induced connection $\nablaC$.}  The coordinate chart $(U,\phi)$
induces Christoffel symbols $\Gamma_{ij}^k$, $i,j,k\in\{1,\ldots,
n\}$, of the Riemannian connection $\nabla$. We now derive the
Christoffel symbols of the induced connection, defined through the
identity
\[
\hat \sigma(\nabla_{V^i} V^j) = \nablaC_{V^i} V^j = \sum_{k=1}^{n-m}
\GammaC_{ij}^k V^k,
\]
where $\GammaC_{ij}^k$ are the symbols we are looking for. Using the
definition of $V^i, V^j$, identity~\eqref{eq:connection_computation},
and the expression for $\hat \sigma$ in~\eqref{eq:sigma:coordinates},
one gets
\[
\GammaC_{ij}^k =\sum_{a=1}^{n} \partial^2_{\th^i \th^j} \varphi^a
+[(\bperp D d\varphi_\th)^{-1} \bperp D]_{ka} \sum_{b,c=1}^{n}
\Gamma_{bc}^a (\partial_{\th^i} \varphi^b) (\partial_{\th^j}
\varphi^c),
\]
where $\varphi^a$ denotes the $a$-th component of $\varphi$. Letting
$\Gamma^a(x)$ be the matrix with components $(\Gamma^a)_{bc}
=\Gamma^a_{bc}$, we may rewrite the Christoffel symbols of the induced
connection in the more economical form
\begin{equation}\label{eq:Christoffel:induced_connection}
\GammaC_{ij}^k =\sum_{a=1}^{n} [(\bperp D d\varphi_\th)^{-1} \bperp
  D]_{ka} \left( \partial^2_{\th^i \th^j} \varphi^a +
(\partial_{\th^i} \varphi)\trans \Gamma^a (\partial_{\th^j}
\varphi)\right) \Big|_{x=\varphi(\th)}, 
\end{equation}
$i,j,k \in \{1,\ldots, n-m\}$.  

\medskip
\noindent
{\bf Constrained dynamics.} The coordinate representation of $\grad
P(q)$ is
\[
\big[ D^{-1}(x) \nabla_x \P(x)\big]_{x = \varphi(s)}.
\]
Using~\eqref{eq:sigma:coordinates}, the projection $\sigma_q (\grad
P(q))$ in $s$-coordinates reads as
\[
\big[(\bperp D d\varphi_s)^{-1} \bperp\nabla_x \P\big]_{x =
  \varphi(s)}.
\]
Letting $e_k$ denote the $k$-th natural basis vector of $\Re^{n-m}$,
the coordinate representation of the constrained
dynamics~\eqref{eq:constrained_dynamics:connection} is
\begin{equation}\label{eq:constrained_dynamics:coordinates}
\ddot \th^k = - \sum_{ij} \GammaC_{ij}^k(\th) \dot \th^i \dot \th^j -
e_k\trans(\bperp D d \varphi_\th)^{-1} \bperp \nabla_x \P \Big|_{x
  = \varphi(\th)}, \ k=1,\ldots, n-m.
\end{equation}
In the special case when $\C$ is diffeomorphic to a generalized
cylinder, one may pick $\varphi$ to be a global diffeomorphism
$(\Se^1)^k \times (\Re)^{n-m-k} \to \C$, in which case the ODEs
in~\eqref{eq:constrained_dynamics:coordinates} constitute a global
representation of the constrained dynamics. In particular, for systems
with degree of underactuation one, i.e., when
$n-m=1$,~\eqref{eq:constrained_dynamics:coordinates} is always valid
globally, and it reduces to
\begin{equation}\label{eq:constrained_dynamics:coordinates:one_dim}
\begin{aligned}
\ddot \th = & -\GammaC_{11}^1(\th) \dot th^2 - \sigma_\th(\grad P(\th)) \\
&=- \frac{\sum_a (\bperp D)_{1a} ( (\varphi^a)'' + \varphi'{}\trans
  \Gamma^a \varphi) } {\bperp D \varphi'}\Big|_{x = \varphi(\th)} \dot
\th^2 -\frac{\bperp \nabla_x \P}{\bperp D \varphi'} \Big|_{x =
  \varphi(\th)},
\end{aligned}
\end{equation}
The above is precisely the form of the constrained dynamics
in~\eqref{eq:constrained_dynamics}-\eqref{eq:Psi_functions}.

\subsection{Examples of computation of constrained dynamics}

We present two examples illustrating the formulas in
Section~\ref{sec:constrained_dynamics:coordinates}.

\begin{example}\label{ex:circle}
Consider the unit point-mass particle on the plane with inertial
coordinates $q=[q_1 \ q_2]\trans \in \Re^2$ depicted in
Figure~\ref{fig:particle}.
\begin{figure}[htb]
\psfrag{q}{$q$}
\psfrag{s}{$\th$} 
\psfrag{a}{$\alpha$}
\psfrag{F}[c]{$R_\alpha q$} 
\psfrag{C}{$\C$}
\centerline{\includegraphics[width=.3\textwidth]{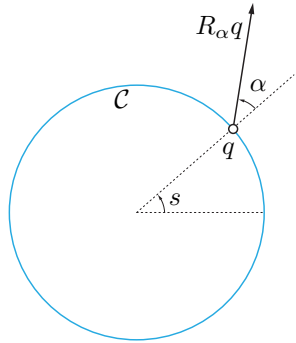}}
\caption{The set $\C$ in Example~\ref{ex:circle} and its parametrization.}
  \label{fig:particle}
\end{figure}

The particle is actuated by a force $(R_\alpha\, q) \tau$, where $\tau
\in \Re$ is the control input, $\alpha$ is a fixed parameter, and
$R_\alpha \in \mathsf{SO}(2)$ is the matrix operating a
counterclockwise rotation by angle $\alpha$.  The equations of motion
are
\begin{equation}\label{eq:EOM:planar_particle}
\ddot q = (R_\alpha \,q) \tau.
\end{equation}
This is a Lagrangian system $(\Q,g,0,F)$, with $\Q=\Re^2$, $g$ the
Euclidean inner product, and $F(q) = (q_1 \cos \alpha - q_2 \sin
\alpha) d q_1 + (q_1 \sin \alpha + q_2 \cos \alpha)dq_2$. Let $\C$ be
the unit circle centred at the origin. If $\alpha \in (-\pi/2,\pi/2)$,
then $\C$ is a regular VHC since, for all $q \in \C$, the vector
$R_\alpha\, q$ is transversal to $\C$:
\[
T_q \C + \spn F^\sharp(q) = \spn \left[\hspace*{-1ex}\begin{array}{r}
    -q_2 \\ q_1
\end{array}\right]
+ \spn\{ R_\alpha\, q\} = T_q \Re^2.
\]
The output function $e = q\trans q -1$ yields vector relative degree
$\{2,2\}$ everywhere on $\C$, and the feedback
\[
\tau^\star(q,\dot q) = \frac{1}{2 q\trans R_\alpha q} (-2 \dot q\trans
\dot q - K_p e - K_d \dot e), \ K_p,K_d>0,
\]
asymptotically stabilizes the constraint manifold $T\C$ and renders it
invariant.  The map $\varphi: \Se^1 \to \Re^2$, $\varphi(\th) = [\cos(\th)
  \ \sin (\th)]\trans$, is a parametrization of $\C$. The Christoffel
symbols $\Gamma_{ij}^k$ of $g$ are all zero. Letting
$\bperp(q):=q\trans R_{\alpha+\pi/2}$ and
using~\eqref{eq:Christoffel:induced_connection}, the Christoffel
symbol of the induced connection is $\GammaC_{11}^1 = \tan \alpha$, so
by~\eqref{eq:constrained_dynamics:coordinates:one_dim} the constrained
dynamics are given by
\[
\ddot \th = -(\tan \alpha) \dot \th^2.
\]
One can also derive the constrained dynamics by multiplying both sides
of~\eqref{eq:EOM:planar_particle} on the left by $\bperp$, and
substituting $q=\varphi(\th)$, $\ddot q = \varphi'(\th) \ddot \th +
\varphi''(\th) \dot \th^2$ in the resulting expression. The ODE one
gets this way is the same as above.
\end{example}

\begin{example}\label{ex:sphere}
Consider now a unit point-mass in $\Re^3$ with inertial coordinates
$q=[q_1 \ q_2 \ q_3]\trans \in \Re^3$, actuated by a control force
$(\diag(1,1,2)q) \tau$, where $\tau \in \Re$ is the control input:
\[
\ddot q = B(q) \tau,\]
where $B(q) = \diag(1,1,2) q$.  This is a Lagrangian system
$(\Q,g,0,F)$, where $\Q = \Re^3$, $g$ is the Euclidean inner product,
and $F(q)=q_1 d q_1 +q_2 d q_2 + 2 q_3 dq_3$. In this example,
$\D_A(q) = \spn\{F^\sharp(q)\} = \image B(q)$.  Let $\C$ be the unit
sphere centred at the origin, $\C =\{q \in \Re^3: q\trans q=1\}$. The
set $\C$ is illustrated in Figure~\ref{fig:sphere}.
\begin{figure}[htb]
\psfrag{a}{$q_1$}
\psfrag{b}{$q_2$}
\psfrag{c}{$q_3$}
\psfrag{C}{$\C$}
\psfrag{1}{$\th^1$}
\psfrag{2}{$\th^2$}
\psfrag{q}{$q$}
\psfrag{F}{$B(q)$}
  \centerline{\includegraphics[width=.35\textwidth]{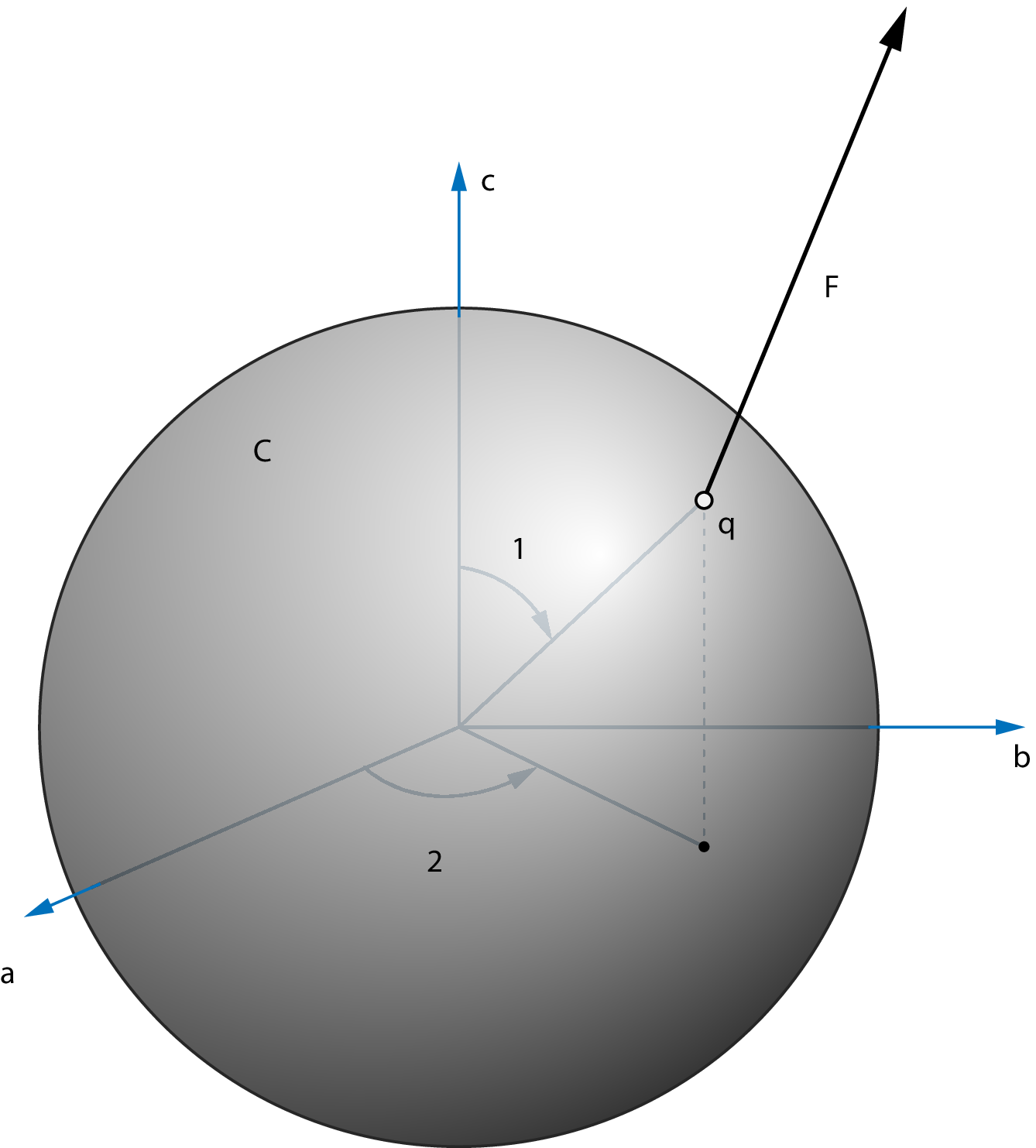}}
  \caption{The VHC $\C$ in Example~\ref{ex:sphere} and its parametrization.}
  \label{fig:sphere}
\end{figure}
For each $q \in \C$, $T_q \C$ is the orthogonal complement of
$\spn\{q\}$. Since $g(q,F^\sharp(q)) = q\trans \diag(1,1,2) q >0$, the
control force is transversal everywhere to the sphere, and therefore,
for any $q \in \C$,
\[
T_q \C \oplus \D_A(q) =T_q \Re^3.
\]
Thus $\C$ is a regular VHC. For a parametrization of $\C$, we use spherical coordinates:
\begin{equation}\label{eq:varphi:sphere}
\varphi(\th^1,\th^2)=\begin{bmatrix} \sin(\th^1) \cos(\th^2)
\\ \sin(\th^1) \sin(\th^2) \\ \cos(\th^1)
\end{bmatrix}.
\end{equation}
Letting $W = (0,\pi) \times (-\pi,\pi)$ and $\hat \C = \Se^2
/\{N,P\}$, where $N$ and $P$ are the north and south poles of $\C$,
the map $\varphi: W \to \hat C$ is a diffeomorphism. To compute the
Christoffel symbols of the induced connection on $\C$, we define a
left-annihilator of $B(q) = \diag(1,1,2) q$:
\[
B^\perp(q)  = \image \begin{bmatrix} -q_2 & q_1 & 0 \\ -q_1 q_3 & -q_2 q_3  & (q_1^2+q_2^2)/2
\end{bmatrix}.
\]
For all $q \in \hat C$, $\rank \bperp(q) =2$ and $\bperp B = 0$, as
required.  Using $\bperp$ above, $\varphi$
in~\eqref{eq:varphi:sphere},  $D=I_3$, and $\Gamma_{ij}^k=0$, we get $\GammaC_{ij}^k$
from~\eqref{eq:Christoffel:induced_connection} as
\[
\begin{array}{lll}
\GammaC_{11}^1 =\frac{\displaystyle -\sin(2 \th^1)}{\displaystyle 2(\cos^2(\th^1)+1)}, &
  \GammaC_{22}^1=\frac{\displaystyle -\sin(2 \th^1)}{\displaystyle \cos^2(\th^1)+1}, &
  \GammaC_{12}^1=\GammaC_{21}^1 = 0, \\
\GammaC_{11}^2=0, & 
\GammaC_{22}^2=0, &  
\GammaC_{12}^2 = \GammaC_{21}^2 = \cot(\th^1). 
\end{array}
\]
Therefore, the coordinate representation of the constrained dynamics
on $T\C$ is given by the ODEs
\begin{equation}\label{eq:constrained_dyn:sphere}
\begin{aligned}
& \ddot \th^1 = \frac{\sin(2 \th^1)}{2(\cos^2(\th^1)+1)} (\dot
  \th^1)^2 + \frac{\sin(2 \th^1)}{\cos^2(\th^1)+1} (\dot \th^2)^2 \\
& \ddot \th^2 = -2 \cot(\th^1) \dot \th^1 \dot \th^2.
\end{aligned}
\end{equation}
In Section~\ref{sec:examples} we will investigate the Lagrangian
structure of~\eqref{eq:constrained_dyn:sphere}.
\end{example}
\section{Existence of a Lagrangian structure for the constrained dynamics}\label{sec:ILP}

In this section we investigate this question: given the Lagrangian
control system~\eqref{eq:Lagrangian_control_system2} and a regular VHC
$\C$ of order $m$, determine whether there exists a Riemannian metric
$g_\C$ on $\C$ and a smooth potential function $P_\C : \C \to \Re$
such that the constrained
dynamics~\eqref{eq:constrained_dynamics:connection} are generated by
the Lagrangian structure $(\C,g_\C, P_\C)$. If this is the case, we
say that {\em the constrained dynamics are Lagrangian.}  The solution
in the special case $m=n-1$ was reviewed in
Theorem~\ref{thm:ILP:1dim}. Here we investigate the problem from a
more general perspective.

\subsection{A general result}

\begin{thm}\label{thm:ILP:general}
If $\C$ is a regular VHC of order $m$ for the Lagrangian control
system~\eqref{eq:Lagrangian_control_system2}, then the constrained
dynamics~\eqref{eq:constrained_dynamics:connection} are Lagrangian if
and only if the following two conditions hold:
\begin{enumerate}[(i)]
  \item The induced connection $\nablaC$ is metrizable, i.e., there
    exists a Riemannian metric $g_\C$ on $\C$ such that $\nablaC$ is
    the Riemannian connection associated with $g_\C$.  %
  \item There exists a smooth function $P_\C: \C \to \Re$ such that
    \[
    \sigma(\grad P) = \gradC P_\C,
    \]
\end{enumerate}
where $\grad_C P_\C \in \X(\C)$ is the gradient vector field of $P_\C$
induced by the metric $g_\C$, i.e., defined by the identity $dP_\C
(v_q) = g_\C (\gradC P_\C,v_q)$ for all $v_q \in T\C$.

Moreover, if (i) and (ii) hold, the Lagrangian structure of the
constrained dynamics is $(\C,g_\C,P_\C)$.
\end{thm}
\begin{proof}

$(\Longleftarrow)$ If conditions (i) and (ii) hold, then it follows
  directly from the definition that the constrained
  dynamics~\eqref{eq:constrained_dynamics:connection} are generated by
  the Lagrangian system $(\C,g_\C,P_\C)$.

\noindent
($\hspace*{-.5ex}\Longrightarrow$) Suppose the constrained
  dynamics~\eqref{eq:constrained_dynamics:connection} are generated by
  a Lagrangian system $(\C,g_\C,P_\C)$. Let $\bar \nabla$ be the
  Riemannian connection associated with $g_\C$.  By
  Theorem~\ref{thm:reduced_dynamics}, a curve $q: I \to \C$ satisfies
\begin{equation}\label{eq:comparison:1}
\nablaC_{\dot q} \dot q + \sigma_q (\grad P(q))=0
\end{equation}
if and only if it satisfies
\begin{equation}\label{eq:comparison:2}
\bar \nabla_{\dot q} \dot q + \gradC P_\C(q) =0.
\end{equation}
For any $q_0 \in \C$, let $q: I \to \C$ be the maximal integral curve
of the constrained dynamics~\eqref{eq:constrained_dynamics:connection}
with initial condition $(q_0,0)$. If $\{X_1,\ldots,X_{n-m}\}$ is any
local frame for $T\C$ defined in a neighbourhood of $q_0$,
identity~\eqref{eq:acceleration} and the fact that $\dot q|_{t=0} =0$
imply that
\[
\nablaC_{\dot q} \dot q \big|_{t=0} = \bar \nabla_{\dot q} \dot q
\big|_{t=0}.
\]
Since~\eqref{eq:comparison:1} and~\eqref{eq:comparison:2} hold, we
deduce that
\[
\sigma_{q_0} (\grad P(q_0)) =\gradC(P(q_0)),
\]
proving that $P_\C$ satisfies condition (ii).

Next, subtracting~\eqref{eq:comparison:1} from~\eqref{eq:comparison:2}
we get
\[
\nablaC_{\dot q} \dot q = \bar \nabla_{\dot q} \dot q.
\]
Since symmetric connections having the same geodesics are equal (see,
instance,~\cite[Theorem~2.101]{Poo07}), the above implies that
$\nablaC=\bar \nabla$. Hence, $\nablaC$ is metrizable, which proves
condition (i).
%% Consider an arbitrary coordinate chart on $\C$ with associated local
%% frame $\{X_1,\ldots,X_{n-m}\}$. Using once again
%% identity~\eqref{eq:acceleration}, for any base integral curve
%% $\gamma:I \to \C$ with coordinate representation
%% $(\gamma^1,\ldots,\gamma^{n-m})$ we have
%% %
%% %
%% %
%% \[
%% \sum_k \left( \ddot \gamma^k + \sum_{i,j} \GammaC_{ij}^k \dot \gamma^i
%% \dot \gamma^k \right) X_k = \sum_k \left( \ddot \gamma^k + \sum_{i,j}
%% \bar\Gamma_{ij}^k \dot \gamma^i \dot \gamma^k \right) X_k,
%% \]
%% %
%% %
%% %
%% where $\bar \Gamma_{ij}^k$ are the Christoffel symbols of $\bar
%% \nabla$. We deduce that, for each $k \in \{1,\ldots, n-m\}$,
%% %
%% %
%% %
%% \begin{equation}\label{eq:quadratic_forms}
%% \sum_{i,j} \GammaC_{ij}^k \dot \gamma^i \dot \gamma^k = \sum_{i,j}
%% \bar\Gamma_{ij}^k \dot \gamma^i \dot \gamma^k.
%% \end{equation}
%% %
%% %
%% %
%% Since $\bar \nabla$ is Riemannian, $\bar \Gamma_{ij}^k =
%% \bar\Gamma_{ji}^k$. By Theorem~\ref{thm:induced_connection}, $\nablaC$
%% is symmetric and thus $\GammaC_{ij}^k
%% =\GammaC_{ji}^k$. Equation~\eqref{eq:quadratic_forms}, then, expresses
%% the equality of two symmetric quadratic forms, which can only hold if
%% $\GammaC_{ij}^k = \bar\Gamma_{ij}^k$. In turn, the equality of the
%% Christoffel coefficients implies the equality of the two connections
%% $\nablaC$ and $\bar \nabla$ on the chart domain. Since the coordinate
%% chart is arbitrary, we conclude that $\nablaC = \bar \nabla$, and
%% hence $\nablaC$ is metrizable. This proves condition (i).
%% }
\end{proof}

\subsection{Case of orthogonal control accelerations}
Referring to the regularity
condition~\eqref{eq:transversality_condition}, when the acceleration
distribution $\D_A$ is fibrewise orthogonal to $T\C$ (see
Figure~\ref{fig:transversality:orthogonal}), the feedback rendering
$T\C$ invariant produces a control force that does no work on base
integral curves contained in $\C$.  
\begin{figure}[htb]
\psfrag{T}{$T_q \Q$}
\psfrag{U}{$T_q \C$}
\psfrag{D}{$\D_A(q)$}
\psfrag{q}{$q$}
\psfrag{C}{$\C$}
\psfrag{Q}{$\Q$}
\centerline{\includegraphics[width=.7\textwidth]{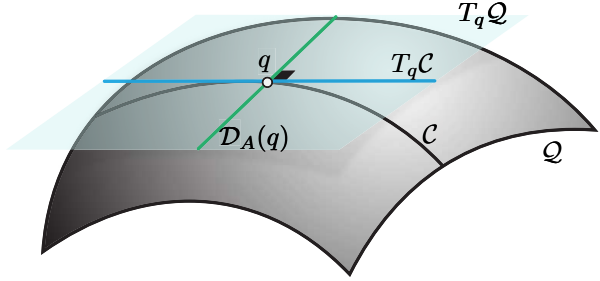}}
\caption{Illustration of the case when the control accelerations are orthogonal to $\C$.} %
\label{fig:transversality:orthogonal}
\end{figure}
In this setting, the control force
is identical to the constraint force that would arise if $\C$ were a
holonomic constraint\footnote{In classical mechanics, holonomic
  constraints such that the constraint force does no work along
  constrained solutions are called {\em ideal.}}.  Just like in
classical mechanics, one expects the constrained dynamics to be
Lagrangian, with Lagrangian structure given by the restriction of the
original Lagrangian structure to $\C$. The next proposition makes this
intuition precise. Recall the inclusion map $\iota: \C \to \Q$. The
metric $g: T \Q \times T \Q \to \Re$ on $\Q$ gives rise to a metric on
$\C$ via the pullback
\[
\iota^* g (v_q,w_q) = g(d \iota_q (v_q),d\iota_q (w_q)) \text{ for all
} v_q,w_q \in T_q \C.
\] 
The metric $\iota^* g$ is called the {\em induced metric on $\C$}.

\begin{prop}\label{prop:orthogonal_forces}
If $\C$ is a regular VHC of order $m$ for the Lagrangian control
system~\eqref{eq:Lagrangian_control_system2} such that
\begin{equation}\label{eq:orthogonal_regularity}
(\forall q \in \C) \ T_q \C \overset{\perp}{\oplus} \D_A(q) = T_q \Q,
\end{equation}
with orthogonality holding with respect to the metric $g$, then the
constrained dynamics~\eqref{eq:constrained_dynamics:connection} are
Lagrangian with Lagrangian structure $(\C,g_\C,P_\C)$, where $g_\C =
\iota^* g$ and $P_\C = P|_\C= P \circ \iota$.
\end{prop}
\begin{proof}
Since $\nabla$ is a Riemannian connection, it satisfies
\begin{equation}\label{eq:compatibility2}
X(g(Y,Z)) = g(\nabla_X Y,Z) + g(Y,\nabla_X Z)
\end{equation}
for all $X,Y,Z \in \X(\Q)$, and therefore also for all $X,Y,Z \in
\X(\C)$. Let $X,Y,Z \in \X(\C)$ be arbitrary.  In light of the
regularity condition~\eqref{eq:transversality:coordinates}, we have
\[
\nabla_X Y = \sigma(\nabla_X Y) + N_X Y = \nablaC_X Y + N_X Y,
\]
where $N_X Y$ is a vector field in the control distribution
$\spn\{(F^i)^\sharp,i=1,\ldots,m\}$. By the orthogonality hypothesis,
we have $g(N_X Y,Z) =0$ for all $Z \in \X(\C)$, implying that
\begin{equation}\label{eq:orthogonality_consequence1}
g(\nabla_X Y,Z) = g(\nablaC_X Y,Z) = g_\C(\nablaC_X Y,Z).
\end{equation}
The second identity in~\eqref{eq:orthogonality_consequence1} is due to
the fact that $\nablaC_X Y$ and $Z$ are vector fields on $\C$.
Analogously to~\eqref{eq:orthogonality_consequence1}, we have
\begin{equation}\label{eq:orthogonality_consequence2}
g(Y,\nabla_X Z) = g_\C(Y,\nablaC_X Z).
\end{equation}
Substituting~\eqref{eq:orthogonality_consequence1}
and~\eqref{eq:orthogonality_consequence2}
into~\eqref{eq:compatibility2} and using the fact that $g(Y,Z) =
g_\C(Y,Z)$, we get
\[
X(g_\C(Y,Z)) = g_\C(\nablaC_X Y,Z) + g_\C(Y,\nablaC_X Z),
\]
implying that $\nablaC$ is compatible with $g_\C$.  Since, by
Theorem~\ref{thm:induced_connection}, $\nablaC$ is symmetric,
$\nablaC$ is Riemannian with respect $g_\C$, proving that condition
(i) of Theorem~\ref{thm:ILP:general} holds.

Next, we need to show that, for each $q \in \C$, $\sigma_q (\grad
P(q)) = \gradC P|_\C$, where $\gradC P|_\C$ is the gradient vector
field of $P|_\C$ induced by the metric $g_\C$.  Since, by assumption,
the subspace $\D_A(q)$ is orthogonal to $T_q \C$, the projection
$\sigma_q : T_q \Q \to T_q \C$ along $\D_A(q)$ is a map whose kernel,
$\D_A(q)$, is orthogonal to its image, $T_q \C$. This fact implies
that $\sigma_q$ is a self-adjoint map. Thus, for all $v_q \in T_q \C$,
\begin{equation}\label{eq:dummy1}
g( \sigma_q (\grad P(q)), v_q ) = g(\grad P(q),\sigma_q(v_q)) =
g(\grad P(q),v_q).
\end{equation}
Using the definition of $\grad$, we have
\begin{equation}\label{eq:dummy2}
g( \grad P(q), v_q ) = dP_q(v_q).
\end{equation}
Since $P|_\C = P \circ \iota$ and, for all $v_q \in T_q \C$, $d
\iota_q(v_q)=v_q$, we may write
\begin{equation}\label{eq:dummy3}
dP_q (v_q) = dP_q \circ d\iota_q (v_q) = d(P\circ \iota)_q (v_q) =
(dP|_\C)_q (v_q).
\end{equation}
Substituting~\eqref{eq:dummy2} and~\eqref{eq:dummy3}
into~\eqref{eq:dummy1} and using the definition of $\gradC$, we get
\[
g( \sigma_q (\grad P(q)), v_q ) = (dP|_\C)_q (v_q) = g_\C( \gradC
P|_\C(q),v_q),
\]
for all $v_q \in T_q \C$. Since $\sigma_q (\grad P(q))$ and $v_q$ lie
in $T_q \C$, $g( \sigma_q (\grad P(q)), v_q ) = g_\C( \sigma_q (\grad
P(q)), v_q )$, and so
\[
g_\C( \sigma_q (\grad P(q)), v_q ) = g_\C( \gradC P|_\C(q),v_q),
\]
for all $v_q \in T_q \C$.  In conclusion, $\sigma_q (\grad P(q)) =
\gradC P|_\C(q)$, proving that $P|_\C$ satisfies condition (ii) of
Theorem~\ref{thm:ILP:general}.  
\end{proof}
\begin{rem}
As mentioned earlier, the foregoing proposition states that the
constrained dynamics associated with a VHC satisfying
condition~\eqref{eq:orthogonal_regularity} coincide with the
constrained dynamics one would have if the Lagrangian system
$(\Q,g,P)$ (without control) were subjected to an ideal holonomic
constraint. The result in Proposition~\ref{prop:orthogonal_forces} is
not new once placed in the context of geometric mechanics, and we do
not claim it to be original.  For
instance,~\cite[Theorem~2.7]{GodNat14} states an analogous result. The
value of Proposition~\ref{prop:orthogonal_forces} lies in the that it
connects the concept of virtual holonomic constraint in control theory
with the concept of ideal holonomic constraint in mechanics in the
special case when the control accelerations are fibrewise orthogonal
to $T\C$. To gain further understanding of the relationship between
Proposition~\ref{prop:orthogonal_forces} and established concepts in
geometric mechanics, it is worth comparing it with Proposition 4.97
in~\cite{geomControlBullo}. In~\cite{geomControlBullo}, a holonomic
constraint $\C$ is a maximal integral manifold of a distribution $\D$
representing a linear velocity constraint. This distribution is used
to define a {\em constrained connection} $\nablaD$ on $\Q$.
Proposition 4.97 in~\cite{geomControlBullo} states that the
restriction of $\nablaD$ to $\X(\C) \times \X(\C)$ is the Riemannian
connection of $\iota^* g$, and Proposition 4.85
in~\cite{geomControlBullo} implies that $\nablaD_X Y$ coincides with
our $\nablaC_X Y$ for all $X,Y \in \X(\C)$.  Taken together, these
results recover the proof of the first part of
Proposition~\ref{prop:orthogonal_forces}, illustrating the strong
analogy between VHCs satisfying
condition~\eqref{eq:orthogonal_regularity} and holonomic constraints
in~\cite{geomControlBullo}. As a caveat, we remark that, unlike the
framework in~\cite{geomControlBullo}, we do not require an integrable
distribution to define the induced connection $\nablaC$, for $\nablaC$
is only required to be defined on $\X(\C) \times \X(\C)$.
\end{rem}

\section{Conditions for metrizability of affine connections}\label{sec:metrizability}

Theorem~\ref{thm:ILP:general} establishes that in the absence of a
potential function, assessing whether or not the constrained dynamics
induced by a regular VHC are Lagrangian amounts to assessing the
metrizability of the induced connection. In this section we review the
main results on metrizability of affine connections, presenting
concrete results for the cases $\dim \C =1$ (already covered in
Theorem~\ref{thm:ILP:1dim}) and $\dim \C=2$. To keep the notation
simple, throughout the section we will consider a symmetric affine
connection $\nabla : \X(\C) \times \X(\C) \to \X(\C)$ with the
understanding that all result will apply to the induced connection
$\nablaC$. We also assume throughout that the submanifold $\C$ is
connected.

\subsection{The holonomy group of an affine connection}
%
%
%Fix a point $q \in \C$, and let $\gamma: I \to \C$, $0\in I$, be a
%xscdfpiecewise smooth closed curve such that $\gamma(0) =q$. 
%
A vector field $X(t)$ along a smooth curve $\gamma$ on $\C$ is said to
be {\em parallel} if its covariant derivative along $\gamma$ vanishes,
i.e., $D_t X \equiv 0$. Given a point $q \in \C$ and a tangent vector
$v_q \in T_q \C$, the equation $D_t X=0$ with initial conditions
$X(0)=v_q$ uniquely determines a parallel vector field $X(t)$ along
$\gamma$ (see~\cite[Chapter II, Proposition 3.3]{book:598491}). This
vector field is called the {\em parallel translation of $v_q$ along
  $\gamma$}.  In local coordinates, the equation $D_t X=0$ is the
linear time-varying ODE
\begin{equation}\label{eq:parallel_transport:coordinates}
\dot X^k = - \sum_{i,j}\dot \gamma^i(t) \Gamma_{ij}^k (\gamma(t)) X^j,
\ k=1,\ldots, n-m,
\end{equation}
where $\Gamma_{ij}^k$ are the Christoffel symbols of $\nabla$ and
$(X^1,\ldots, X^{n-m})$ is the coordinate representation of $X$.

In what follows, if $q,p \in \C$, a piecewise smooth curve in $\C$
starting at $q$ and ending at $p$ will be denoted $\gamma_q^p$. More
precisely, $\gamma_q^p:[0,T]\to \C$ will be a piecewise smooth map
such that $\gamma_q^p(0)=q$ and $\gamma_q^p(T)=p$. On the other hand,
a {\em loop at $q$}, i.e., a piecewise smooth closed curve through $q$
will be denoted by $\gamma_q$. The set of all loops at $q$ will be
denoted by $L_q$.

For a curve $\gamma_q^p$, the {\em parallel transport map along
  $\gamma_q^p$}, denoted $\PP_{\gamma_q^p} :T_q \C \to T_p \C$, is
defined as $\PP_{\gamma_q^p} (v_q) :=X(T)$, where $X$ is the parallel
translation of $v_q$ along $\gamma_q^p$. For $\gamma_q \in L_q$,
$\PP_{\gamma_q}$ maps $T_q \C$ onto itself.  The parallel transport map
associated with a loop is illustrated in Figure~\ref{fig:holonomy}.
\begin{figure}[htb]
\psfrag{q}[c]{\small $q$} 
\psfrag{v}[c]{\small $v_q$}
\psfrag{w}[c]{\small $\PP_{\gamma_q}(v_q)$} 
\psfrag{g}[c]{\small $\gamma_q$}
\psfrag{T}{\small $T_q \C$}
\centerline{\includegraphics[width=.6\textwidth]{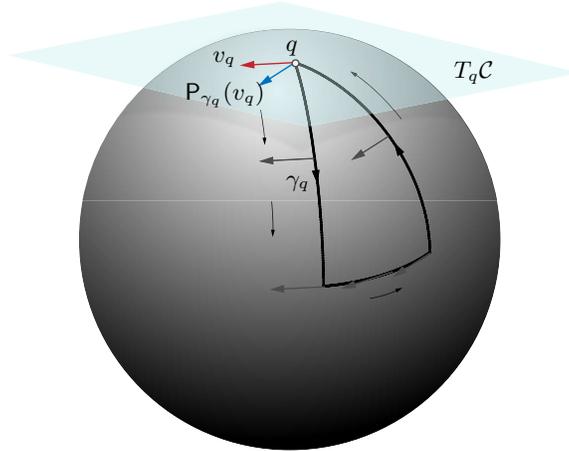}}
\caption{The parallel transport map at the north pole of the unit
  sphere in $\Re^3$, with Riemannian connection induced by the
  Euclidean metric in $\Re^3$. The loop $\gamma_q$ is a triangle on
  the sphere.}
  \label{fig:holonomy}
\end{figure}

Denoting by $(\gamma_q^p)^{-1}$ the curve obtained by reversing the
orientation of $\gamma_q^p$, and by $\gamma_q^p \cdot \gamma_p^r$ the
concatenation of $\gamma_q^p$ with $\gamma_p^r$, we have the following
result.

\begin{prop}[\cite{kobayashi1996foundations}, Chapter II, Proposition 3.3]
\label{prop:parallel_transport:group_property}
For each $q \in \C$ and any piecewise smooth curves $\gamma_q^p$,
$\gamma_p^r$, the parallel transport map $\PP_{\gamma_q^p}:T_q \C \to
T_p \C$ is an isomorphism enjoying the following properties:

\begin{enumerate}[(i)]

\item If $\gamma_q \in L_q$ is the constant loop $\gamma_q(t)\equiv
  q$, then $\PP_{\gamma_q}$ is the identity map on $T_q \C$.

\item $\PP_{(\gamma_q^p)^{-1}} = (\PP_{\gamma_q^p})^{-1}$.

\item $\PP_{\gamma_q^p \cdot \,\gamma_p^r} = \PP_{\gamma_q^p} \circ
  \PP_{\gamma_p^r}$.

\end{enumerate}

In particular, the set of all isomorphisms $\{\PP_{\gamma_q} :
\gamma_q \in L_q\}$ forms a group under composition.
\end{prop}
The group $\Hol_q(\nabla):=\{\PP_{\gamma_q} : \gamma_q \in L_q\}$ of
all parallel transport maps along loops at $q$ is called the {\em
  holonomy group of $\nabla$ with reference point $q$}, while the
subgroup $\Hol^0_q(\nabla) :=\{\PP_{\gamma_q} : \gamma_q \in L_q
\text{ is contractible to } q\}$ is the {\em restricted holonomy
  group}.
%For
%any two points $p, q \in \C$, the holonomy groups $\Hol_p(\nabla)$ and
%$\Hol_q(\nabla)$ are isomorphic to each other.  
A remarkable property of the holonomy groups is that they possess a Lie group
structure.
\begin{thm}[\cite{kobayashi1996foundations}, Chapter II, Theorem 4.2]
\label{thm:holonomy:Liesubgroup}
Let $\C$ be a connected manifold, and let $q \in \C$. Then the
following are true:

\begin{enumerate}[(i)]
\item $\Hol_q^0(\nabla)$ is a connected Lie subgroup of $\GL(n)$.

\item $\Hol_q(\nabla)$ is a Lie subgroup of $\GL(n)$ whose
  identity component is $\Hol^0_q(\nabla)$.
%
%$\Hol^0_q(\nabla)$ is a normal subgroup of $\Hol_q(\nabla)$ and
%  the quotient group $\Hol_q(\nabla) / \Hol_q^0(\nabla)$ is
%  countable.  
\end{enumerate}

\end{thm}

\subsection{Schmidt's metrizability theorem}
The significance of the holonomy group at it pertains to the
metrizability of $\nabla$ rests upon the following consideration. If
$\nabla$ is Riemannian with respect to a metric $g$, then it is a
basic fact of Riemannian geometry that for any two vector fields $V,W$
that are parallel along a curve $\gamma$, $g(V,W)$ is constant along
$\gamma$. In particular, for any $\gamma_q \in \Hol_q (\nabla)$,
$g_q(\PP_{\gamma_q}(v_q),\PP_{\gamma_q}(w_q)) = g_q(v_q,w_q)$. Also,
by definition, $g_q$ is a positive definite, symmetric bilinear form
on $T_q \C$. In conclusion, a necessary condition for $\nabla$ to be
metrizable is that there exists a symmetric, positive definite
bilinear form $T_q \C \times T_q \C \to \Re$ that is invariant under
the holonomy group $\Hol_q(\nabla)$. This condition is also
sufficient.
\begin{thm}[\cite{Schmidt73}] \label{thm:holonomy:schmidt}
Let $\nabla$ be a symmetric affine connection on a connected manifold
$\C$ and let $q \in \C$ be arbitrary. Then $\nabla$ is metrizable if
and only there exists a symmetric positive definite bilinear form $g_q
:T_q \C \times T_q \C \to \Re$ that is invariant under
$\Hol_q(\nabla)$, i.e., for all $\gamma_q \in \Hol_q(\nabla)$ and all
$ v_q,w_q \in T_q \C$,
\begin{equation}\label{eq:holonomy_invariance}
g_q (\PP_{\gamma_q}(v_q),\PP_{\gamma_q}(w_q) ) = g_q(v_q,w_q).
\end{equation}
\end{thm}
Schmidt's theorem only requires one to determine whether or not a
bilinear form on $T_q \C \times T_q \C$ exists which is invariant
under $\Hol_q(\gamma_q)$. It then guarantees that the form in question
can be extended to a Riemannian metric defined on the whole of $T\C
\times T\C$.  We have already outlined the necessity part of the
proof.  The idea of the sufficiency proof is to extend the bilinear
form $g_q$ to the entire tangent bundle $T \C$ by parallel translation
along curves connecting $q$ to arbitrary points in $\C$. Specifically,
for arbitrary $p \in \C$, pick an arbitrary piecewise smooth
$\gamma_q^p$ connecting $p$ and $q$, and define
\begin{equation}\label{eq:metric:parallel_translation}
g_p(v_p,w_p):=g_q\big(\PP_{\gamma_q^p}(v_p),\PP_{\gamma_q^p}(w_p)\big).
\end{equation}
The invariance of $g_q$ under $\Hol_q(\nabla)$ guarantees that $g_p$
is path-independent, giving rise to a Riemannian metric on $\C$.  One
can  easily show that the extension so obtained is a Riemannian
metric associated with $\nabla$.

The result in Theorem~\ref{thm:holonomy:schmidt} is of difficult
application because the group $\Hol_q(\nabla)$ is generally hard to
find. The Ambrose-Singer theorem~\cite{Amb_Sin_1953} characterizes
$\Hol^0_q(\nabla)$ in terms of the curvature form of the connection,
but it requires the knowledge of the so-called holonomy bundle, an
object which is not readily available. In special cases, however, the
computations are more manageable, as we discuss next.

\subsection{Flat connections} 
The {\em curvature endomorphism induced by an affine connection
  $\nabla$} is the map $\X(\C) \times \X(\C) \times \X(\C) \to \X(\C)$
defined as
\[
R(X,Y) Z :=\nabla_X \nabla_Y Z - \nabla_Y \nabla_X Z - \nabla_{[X,Y]}
Z.
\]
If $\nabla$ is a flat connection, i.e., the curvature $R$ induced by
$\nabla$ is zero, then the Ambrose-Singer theorem implies that
$\Hol_q^0(\nabla)$ is trivial, and by~\cite[Theorem 2]{AusMar55} there
exists a surjective homomorphism $\pi_1(\C,q) \to \Hol_q(\nabla)$,
where $\pi_1(\C,q)$ is the first homotopy group of $\C$ with reference
point $q$. The homomorphism in question sends an equivalence class of
loops $[\gamma_q] \in \pi_1(\C,q)$ to a parallel transport map
$\PP_{\gamma_q} \in \Hol_q(\nabla)$.  In this case, to apply
Theorem~\ref{thm:holonomy:schmidt} it suffices to compute the
transport maps associated with the generators of $\pi_1(\C,q)$, as
stated next.

\begin{prop}\label{prop:holonomy:flat}
Let $\nabla$ be a symmetric affine connection on a connected manifold
$\C$, and suppose that $\nabla$ is flat. Let $q \in \C$ be arbitrary,
and let $S_q$ be a set of generators of $\pi_1(\C,q)$. Then $\nabla$ is
metrizable if and only if there exists a symmetric positive definite
bilinear form $g_q : T_q \C \times T_q \C \to \Re$ such that, for each
equivalence class $E_q \in S_q$, there exist a piecewise smooth curve
$\gamma_q \in E_q$ for which
\[
g_q (\PP_{\gamma_q}(v_q), \PP_{\gamma_q}(w_q)) = g_q(v_q,w_q),
\]
for any $v_q,w_q \in T_q\C$.
\end{prop}
\subsection{Simply connected manifolds} 
When $\C$ is simply connected, $\Hol_q(\nabla) = \Hol_q^0(\nabla)$
because, by definition, all loops at $q$ in $\C$ are contractible to
$q$. By Theorem~\ref{thm:holonomy:Liesubgroup}, $\Hol_q^0(\nabla)$ is
a connected Lie group, implying that it is entirely characterized by
its Lie algebra, the so-called {\em holonomy algebra of $\nabla$}. We
will denote by $\hh$ the holonomy algebra. For simply connected
manifolds, one may express the invariance condition
in~\eqref{eq:holonomy_invariance} in infinitesimal form, giving rise
to a Lie algebraic metrizability criterion.
\begin{lem}[\cite{Van08}]\label{lem:holonomy:Lie_algebra_criterion}
Let $\nabla$ be a symmetric affine connection on a simply connected
manifold $\C$ and let $q \in \C$ be arbitrary.  A symmetric positive
definite bilinear form $g_q: T_q \C \times T_q \C \to \Re$ is
invariant under $\Hol_q(\nabla)$ if and only if for all $A \in \hh$
and all $v_q,w_q \in T_q \C$,
\begin{equation}\label{eq:holonomy_invariance:infinitesimal}
g_q(A v_q, w_q) + g_q(v_q, A w_q) =0.
\end{equation}
\end{lem}
\begin{rem}\label{rem:metrizability:simply_connected}
If $\nabla$ is real analytic\footnote{And in other cases when $\nabla$
  is smooth, see~\cite{kobayashi1996foundations}.}, the holonomy
algebra $\hh$ is entirely characterized by the curvature and its
covariant derivatives (see~\cite[Chapter II, Proposition 10.4 and
  Theorem 10.8]{kobayashi1996foundations}). Therefore, $\hh$ can be
computed in local coordinates. Then, a consequence of
Lemma~\ref{lem:holonomy:Lie_algebra_criterion} is that if $\C$
is simply connected and $\nabla$ is a real analytic affine connection
on $\C$,  $\nabla$ is metrizable if and only if it is locally
metrizable (i.e., metrizable in local coordinates).
\end{rem}
Exploiting Lemma~\ref{lem:holonomy:Lie_algebra_criterion} and the de
Rham decomposition of a Riemannian manifold, Kovalski in~\cite{Kow88}
gave an effective decision algorithm for metrizability of real
analytic affine connections on simply connected manifolds. We will not
review the algorithm here, but we refer the reader to the review
in~\cite{Van08}.

\subsection{One-dimensional manifolds}\label{sec:one_dim_manifolds}

If $\C$ is one-dimensional, then it is diffeomorphic to either $\Re$
or $\Se^1$. This situation occurs in Lagrangian control systems with
degree of underactuation one, when a regular VHC of codimension one is
enforced. In Theorem~\ref{thm:ILP:1dim} of
Section~\ref{sec:background_vhc}, we reviewed necessary and sufficient
conditions for the constrained dynamics to be Lagrangian.  Now we show
that the conditions of Theorem~\ref{thm:ILP:1dim} have an elegant
interpretation in the context of induced connections. We will recover
Theorem~\ref{thm:ILP:1dim} as a corollary of
Theorem~\ref{thm:ILP:general} and
Proposition~\ref{prop:holonomy:flat}.

We begin with the observation that any affine connection on a
one-dimensional manifold is flat, so if $\dim \C =1$, we may apply
Proposition~\ref{prop:holonomy:flat} to assess the metrizability of
the induced connection.  By comparing~\eqref{eq:Psi_functions}
and~\eqref{eq:constrained_dynamics:coordinates:one_dim}, we deduce
that 
\begin{equation}\label{eq:one_dim:geometric_data}
\GammaC_{11}^1(\th) = -\Psi_2(\th), \ \sigma_s(\grad P(\th))=-\Psi_1(\th),
\ \ \th \in \Theta
\end{equation}
where we recall that $\Theta$, defined in
Proposition~\ref{prop:reduced_dynamics:coordinates}, is $\Re$ or
$\Se^1$ depending on whether $\C \simeq \Re$ or $\C \simeq \Se^1$.

If $\C \simeq \Re$, then $\pi_1(\C,q)$ is trivial, and
Proposition~\ref{prop:holonomy:flat} is trivially satisfied. Thus a
connection on $\Re$ is always metrizable. Using $x \in \Re$ as
coordinate for $\C$, the Riemannian metric on $\Re$ will have the form
$g_x(v,w) = (1/ 2) k(x) v w$, with $k(x) >0$. By
Theorem~\ref{thm:ILP:general}, the constrained dynamics are Lagrangian
if and only if there exists a function $P_\C : \Re \to \Re$ such that
$\sigma_x(\grad P(x)) = k^{-1}(x) P_\C'(x)$. This identity is
satisfied by letting $P_\C$ be an antiderivative of the function $k(x)
\sigma_x(\grad P(x))$. Having established the metrizability of the
induced connection and the existence of $P_\C$, by
Theorem~\ref{thm:ILP:general} the constrained dynamics are always
Lagrangian.  This recovers the result of Theorem~\ref{thm:ILP:1dim}
when $\C \simeq \Re$.

Now consider the case $\C \simeq \Se^1$, so that $\Theta = \Se^1$.
Since $\pi_1(\Se^1,0) =(\Ze,+)$, $\pi_1(\Se^1,0)$ is generated by the
loop $\gamma_0: [0 \ 2\pi] \to \Se^1$, $t \mapsto t \Mod 2\pi$.  By
Proposition~\ref{prop:holonomy:flat}, the induced connection is
metrizable if and only if there exists a positive definite quadratic
form that is invariant under the transport map $\PP_{\gamma_0}$.
Recall the coordinate representation of the parallel transport map
in~\eqref{eq:parallel_transport:coordinates}. Using $t \in \Re$ as
local coordinates for $\Se^1$, we have that $\PP_{\gamma_0}(v)=X(2\pi)$,
where $X$ is the solution of the linear time-varying ODE
\[
\begin{aligned}
&\dot X = \big( \Psi_2\circ \pi (t) \big) X \quad t \in [0,2\pi)\\
&X(0) = v.
\end{aligned}
\]
To obtain the above ODE, we substituted the first identity
of~\eqref{eq:one_dim:geometric_data}
into~\eqref{eq:parallel_transport:coordinates}, and used the fact that
the coordinate representation of $\Psi_2(\th)$ is $\Psi_2 \circ \pi(t)$,
where $\pi(t) = t \Mod 2\pi$.  The solution of the above scalar linear
system is
\[
\PP_{\gamma_0}(v) = \left( \exp \int_0^{2\pi} \Psi_2 \circ \pi(z) dz
\right) v.
\]
We pick $\th=0$ as reference point on $\Se^1$. Then, modulo a
multiplicative positive scalar, the only positive definite bilinear
form on $T_0 \Se^1 \times T_0 \Se^1$ is $g_0(v_0,w_0) = v_0 w_0$, and
the invariance condition in Proposition~\ref{prop:holonomy:flat} reads
as
\[
 \exp \left( 2 \int_0^{2\pi} \Psi_2 \circ \pi(z) dz\right) v_0 w_0 =
 v_0 w_0.
\]
The above identity holds for arbitrary $v_0,w_0$ if and only if
\begin{equation}\label{eq:one_dim_metrizability}
\int_0^{2\pi} \Psi_2 \circ \pi(z) dz =0,
\end{equation}
or, equivalently, if the function $\hMC(x)$ in Theorem~\ref{thm:ILP:1dim}
is $2\pi$-periodic. Thus, the periodicity requirement on $\hMC$ in
part~(b) of Theorem~\ref{thm:ILP:1dim} is equivalent to the
requirement, in part (i) of Theorem~\ref{thm:ILP:general}, that the
induced connection be metrizable.

To find the Riemannian metric on $\Se^1$ (denoted $g$ in what
follows), we extend the inner product $g_0 : T_0 \Se^1 \times T_0
\Se^1 \to \Re$ to the whole $T\Se^1 \times T \Se^1$ through parallel
transport, as in~\eqref{eq:metric:parallel_translation}. For any $\th
\in \Se^1$, set $g_\th (v_\th,w_\th)
:=g_0(\PP_{\gamma^0_\th}(v_\th),\PP_{\gamma^0_\th}(w_\th))$, where
$\gamma^0_\th$ is an arbitrary curve from $\th$ to $0$ in $\Se^1$. For
instance, pick any $x \in \pi^{-1}(\th)$, and define $\gamma^0_\th:[0,
  x] \to \Se^1$ as $\gamma^0_\th(t) = \th-(t \Mod 2\pi)$. Then,
$\PP_{\gamma^0_\th}(v_\th) = \exp \big( - \int_0^{x} \Psi_2 \circ
\pi(z) d z\big) v_\th$.  Using $\PP_{\gamma^0_\th}$, the Riemannian
metric on $\Se^1$ is
\[
g_\th(v_\th,w_\th) = \exp\left( -2 \int_0^{\pi^{-1}(\th)} \Psi_2 \circ
\pi(z) d z \right).
\]
The above metric on $\Se^1$ gives  the kinetic energy of the
constrained dynamics in Theorem~\ref{thm:ILP:1dim}, since
it can be expressed as $g_\th(v_\th,w_\th) = \MC(\th) v_\th w_\th$.

Next, we turn our attention to condition (ii) of
Theorem~\ref{thm:ILP:general}, namely the existence of $P_{\C} :
\Se^1 \to \Re$ such that $\sigma(\grad P) = \gradC P_{\C}$,
or
\[
 -\Psi_1(\th) = \frac{1}{\MC(\th)} P_\C'(\th)
\]
Equivalently, we need to check when is it that the one-form on $\Se^1$
$-\Psi_1(\th) \MC(\th) d\th$ is exact.  This is the case if and only if
the integral of the form along $\Se^1$ is zero,
\[
0 = \int_{\Se^1} \Psi_1(\th) \MC (\th) d\th = \int_0^{2 \pi}
\big(\Psi_1 \circ \pi(\tau)\big) \hMC(\tau) d \tau.
\]
This is precisely the condition that the function $\hPC(x)$ in
Theorem~\ref{thm:ILP:1dim} be $2\pi$-periodic. Thus, the periodicity
requirement on $\hPC$ in part (b) of Theorem~\ref{thm:ILP:1dim} is
equivalent to the requirement, in part (ii) of
Theorem~\ref{thm:ILP:general}, that $\sigma(\grad P) = \gradC
P_\C$. We have thus shown that Theorem~\ref{thm:ILP:1dim} is a
corollary of Theorem~\ref{thm:ILP:general} and
Proposition~\ref{prop:holonomy:flat}.

\subsection{Two-dimensional manifolds}

When $\dim \C=2$, the metrizability of an affine connection has a
powerful characterization in terms of the Ricci
tensor~\cite{Thompson_1991,VanzZack_09}. We remark that the
situation $\dim \C =2$ arises in Lagrangian control systems with
degree of underactuation two, when a regular VHC of codimension two is
enforced.  The Ricci curvature tensor (see, e.g.,~\cite{book:598491})
is the $(0,2)$ tensor defined as
\[
\Ric(Y_p,Z_p) = \trace(X_p \mapsto R(X_p,Y_p) Z_p), \ p \in \C,
X_p,Y_p \in T_p \C,
\]
where $\trace(\cdot)$ denotes the trace of a linear map. If the affine
connection $\nabla$ is Riemannian with respect to a metric $g$, then
the Ricci tensor is proportional to the metric, 
\begin{equation}\label{eq:Ricci}
\Ric(X,Y) = g(X,Y)
K,
\end{equation}
where $K \in C^\infty(\C)$ denotes the Gaussian curvature of $\C$
(see~\cite[Lemma 8.7]{book:598491}). Recall
from~\eqref{eq:compatibility:total_der} that the compatibility of
$\nabla$ with $g$ means that $\nabla g=0$. If $\nabla$ has
nonvanishing curvature, then $K$ is nonvanishing, and setting
$\alpha=1/ K$, we have
\begin{equation}\label{eq:nabla_scaled_ricci}
\nabla( \alpha \Ric) = \nabla g =0.
\end{equation}
For each $X,Y,Z \in \X(M)$, we have
\[
\nabla_X (\alpha \Ric(Y,Z)) = X(\alpha) \Ric(Y,Z) + \alpha \nabla_X \Ric(Y,Z),
\]
and using~\eqref{eq:nabla_scaled_ricci} we deduce that
\[
\nabla_X \Ric(Y,Z) = - \frac{ d\alpha(X)}{\alpha}  \Ric(Y,Z) = d (-\ln|\alpha|) (X) \Ric(Y,Z).
\]
The above may be rewritten concisely using the total covariant
derivative and the tensor product as 
\[
\nabla \Ric =  d (-\ln|\alpha|) \otimes \Ric.
\]
A tensor field $F$ whose total covariant derivative satisfies $\nabla
F = \omega \otimes F$, where $\omega$ is a one-form, is said to be
{\em recurrent}. Thus, a necessary condition for metrizability of
$\nabla$ is that the Ricci tensor induced by $\nabla$ be recurrent,
with a one-form $\omega$ given by the exact differential of a function
in $C^\infty(\C)$. Further, in light of the fact that when $\nabla$ is
metrizable identity~\eqref{eq:Ricci} holds, another necessary
condition for metrizability is that the Ricci tensor be definite
(positive definite if $K > 0$, negative definite if $K < 0$).
Together, these conditions are also sufficient.
\begin{thm}[\cite{Thompson_1991},\cite{VanzZack_09}\footnote{In these papers, the authors investigate the existence of non-degenerate metrics, whereas we look for positive definite metrics. The statement of the theorem has been adapted accordingly.}]\label{thm:metrizability:2dim}
Let $\C$ be a two-dimensional connected manifold and $\nabla$ a
symmetric affine connection on $\C$ such that the curvature induced by
$\C$ is nowhere zero. Then $\nabla$ is metrizable if and only if the
Ricci tensor induced by $\nabla$ is definite and recurrent, with the
corresponding one-form being exact. If this is the case, and $\nabla
\Ric = d f \otimes \Ric$ holds for some $f \in C^\infty(\C)$, then all
Riemannian metrics compatible with $\nabla$ are given by
\[
g = \pm \exp(-f +b) \Ric, \ b \in \Re \text{ arbitrary},
\]
with plus sign if $\Ric$ is positive definite, and minus sign otherwise.
\end{thm}
The idea behind the proof of sufficiency rests upon the fact that if
$\nabla \Ric = df \otimes \Ric$, then $\nabla \big( \exp(-f+b) \Ric
\big) =0$, and therefore both type $(0,2)$ tensor fields given by $\pm
\exp(-f+b) \Ric$ are compatible with $\nabla$. Since $\nabla$ is
symmetric, so is $\Ric$. Since $\Ric$ is definite, $g$ in the theorem
statement is positive definite.

\section{Examples}\label{sec:examples}
We now illustrate the results of Sections~\ref{sec:ILP}
and~\ref{sec:metrizability} with three examples.  First, we revisit
Examples~\ref{ex:circle} and~\ref{ex:sphere}. Then we investigate the
Lagrangian structure of a double pendulum on a cart subject to a
regular VHC of order 1.

\setcounter{examp}{0}
\begin{example}[Continued]
Consider again the dynamics of the planar point-mass of
Example~\ref{ex:circle}, a Lagrangian control system with
underactuation degree one.  The constrained dynamics on the unit
circle $\C$ depicted in Figure~\ref{fig:particle} are
\[
\ddot \th = -(\tan \alpha) \dot \th^2, \ s \in \Se^1.
\]
We want to determine the existence of a Lagrangian structure for these
constrained dynamics.  For this, we may use
Theorem~\ref{thm:ILP:1dim}, with $\Psi_1(\th)=0$ and $\Psi_2(s)=-\tan
\alpha$. We have
\[ 
\begin{aligned}
& \hMC(x) = \exp\left(-2 \int_0^x \tan(\alpha) dz \right) = \exp(-2 \tan(\alpha) x), \\
& \hPC(x)=0.
\end{aligned}
\]
The constrained dynamics are Lagrangian if and only if $\hMC$ is $2
\pi$-periodic, or $\alpha = \pm \pi/2 \Mod 2\pi$. Thus the constrained
point-mass is Lagrangian if and only if the control force is
orthogonal to the circle $\C$, in which case the Lagrangian is
$L(\th,\dot \th)=(1/2) \dot \th^2$. As predicted by
Proposition~\ref{prop:orthogonal_forces}, $L(\th,\dot \th)$ is the
restriction of the original Lagrangian to $\C$, i.e., $L(\th,\dot \th)
= (1/2) g(\dot q,\dot q)|_{\dot q = \varphi'(\th)\dot \th}$.
\end{example}
\begin{example}[Continued]
We return now to the unit mass of Example~\ref{ex:sphere}, a
Lagrangian control system with degree of underactuation two. We seek
to determine whether or not the constrained dynamics with coordinate
representation given in~\eqref{eq:constrained_dyn:sphere} are
Lagrangian. Since this Lagrangian control system has no potential
function, by Theorem~\ref{thm:ILP:general} we only need to check whether or
not the induced connection on $\C$ is metrizable. To this end, we will
use Theorem~\ref{thm:metrizability:2dim}. Since $\C$ is simply
connected and the induced connection is real analytic, it suffices to check the
recurrence condition of Theorem~\ref{thm:metrizability:2dim} in local
coordinates (see Remark~\ref{rem:metrizability:simply_connected}). 

In local coordinates $(\th^1,\th^2) \in W \subset \Re^2$, we have the
frame $\{\partial_1,\partial_2\}$ given by the natural basis of
$\Re^2$.  We first compute the coefficients $R_{ijk}^l$ of the
curvature endomorphism associated with $\nablaC$ in
$(\th^1,\th^2)$-coordinates via the formula\footnote{Here we use the
  fact that $[\partial_i,\partial_j]=0$.}  $R_{ijk}^l=d\th^l \left(
\nablaC_{\partial_i} \nablaC_{\partial_j} \partial_k -
\nablaC_{\partial_j} \nablaC_{\partial_i} \partial_k\right)$, $i,j,k,l
\in \{1,2\}$.  Using~\eqref{eq:connection_computation} for the
evaluation of $\nablaC$ we obtain
\begin{equation} \label{eq:curvature_coordinates}
R_{ijk}^l=\partial_{\th^i} \GammaC_{jk}^l - \partial_{\th^j}
\GammaC_{ik}^l + \sum_m \big(\GammaC_{jk}^m \GammaC_{im}^l -
\GammaC_{ik}^m \GammaC_{jm}^l \big),
\end{equation}
where $\GammaC_{ij}^k$ are the Christoffel symbols
in~\eqref{eq:Christoffel:induced_connection}.

The coefficients
$\Ric_{ij}$ of the Ricci tensor are then given by $\Ric_{ij} = \sum_k
R_{kij}^k$. Performing these computations, we get
\[
\Ric_{11} = \frac{1}{\cos^2(\th^1)+1}, \ \Ric_{12}=\Ric_{21} = 0, \  \Ric_{22}= \frac{2
  \sin^2(\th^1)}{(\sin^2(\th^1) - 2)^2}.
\]
Next, using the total covariant differentiation of
tensors,~\eqref{eq:covariant_derivative:tensor},~\eqref{eq:total_covariant_derivative},
and the Christoffel symbols~\eqref{eq:Christoffel:induced_connection},
we compute the coefficients $(\nablaC\Ric)_{ijk}$ of $\nablaC \Ric$ by
means of
\[
(\nablaC\Ric)_{ijk} =\nablaC_{\partial_i} \Ric(\partial_j,\partial_k). 
\]
By so doing, we find that the only nonzero coefficients
$(\nablaC\Ric)_{ijk}$ are
\begin{equation}\label{eq:nablaRic_sphere}
(\nablaC \Ric)_{111} = \frac{\displaystyle 2 \sin(2
   \th^1)}{\displaystyle (\cos^2(\th^1)+ 1)^2}, 
\ 
(\nablaC \Ric)_{122} =-\frac{\displaystyle 4 \sin(2 \th^1)
    \sin^2(\th^1)}{\displaystyle (\sin^2(\th^1) - 2)^3}.
\end{equation}
Next, we check whether or not $\nablaC \Ric = df \otimes \Ric$ for a
suitable smooth function $f$. Consider a generic one-form $\omega$ on
$W \subset \Re^2$, $\omega = \omega_1 d\th^1 + \omega_2 d\th^2$. The
nonzero coefficients $(\omega \otimes \Ric)_{ijk}$ of the tensor
product $\omega \otimes \Ric$ are \newcommand{\tp}{(\omega \otimes
  \Ric)}
\begin{equation}\label{eq:recurrence_sphere}
\begin{array}{ll}
\tp_{111}=\frac{\displaystyle\omega_1}{\displaystyle \cos^2(\th^1)+
  1}, 
& \tp_{122} = \frac{\displaystyle 2 \omega_1
  \sin^2(\th^1)}{\displaystyle(\sin^2(\th^1)- 2)^2}, \\[2ex]
\tp_{211} = \frac{\displaystyle\omega_2}{\displaystyle \cos^2(\th^1)+
  1}, 
& \tp_{222} = \frac{\displaystyle 2 \omega_2
  \sin^2(\th^1)}{\displaystyle(\sin^2(\th^1)- 2)^2}.
\end{array}
\end{equation}
By comparing~\eqref{eq:nablaRic_sphere}
and~\eqref{eq:recurrence_sphere}, we see that $\nablaC \Ric = \omega
\otimes \Ric$ if and only if $\omega_1 = 2 \sin(2 \th^1) / (
\cos^2(\th^1) +1)$ and $\omega_2=0$, i.e.,
\[
\omega = \frac{2 \sin(2 \th^1)}{\cos^2(\th^1)+1} d \th^1.
\]
The one-form $\omega$ is exact, $\omega = df$, with $f(s) = -4
\atanh(\sin^2(\th^1)/(\sin^2(\th^1)-4))$.  By
Theorem~\ref{thm:metrizability:2dim}, the constrained dynamics on $T\C$
are Lagrangian, and a Lagrangian function in local coordinates is
given by
\[
L(s,\dot s) = (1/2)\dot s\trans D_\C(s) \dot s,
\]
where $D_\C(s)$ is the matrix $\exp(-f(s)) [\Ric]$, with $[\cdot]$
denoting the matrix representation of the Ricci tensor. Specifically,
we have
\[
D_\C(s)=\begin{bmatrix} 1/2-\sin^2(\th^1)/4 & 0 \\  0 & \sin^2(\th^1)/2
\end{bmatrix}.
\]
In applying Theorem~\ref{thm:metrizability:2dim}, we used the plus
sign in the metric because the matrix $[\Ric]$ is positive
definite. One may check that the Euler-Lagrange equation with $L$ as
above gives the constrained
dynamics~\eqref{eq:constrained_dyn:sphere}. We stress once again that
although our computations are done in local coordinates, by the
argument in Remark~\ref{rem:metrizability:simply_connected} the
foregoing considerations imply that the constrained dynamics are {\em
  globally Lagrangian.}
\end{example}

\begin{example}\label{ex:double_pendulum_cart}
Consider the double pendulum on a cart depicted in
Figure~\ref{fig:double_pendulum_cart}, a Lagrangian control system
with three degrees-of-freedom and one input. We investigate two cases.
\begin{itemize}
\item {\bf Case (a):} the control input is the force imparted to the cart. 
\item {\bf Case (b):}
the control input is the torque imparted on the second revolute joint.
\end{itemize}
 We
assume that the pendulum rods are massless, that the masses of the two
pendulums are unitary, and that the rod lengths are unitary as well.
\begin{figure}[htb]
\psfrag{a}{$q_1$}
\psfrag{b}{$q_2$}
\psfrag{c}{$q_3$}
\psfrag{t}{$\tau$}
\psfrag{A}{(a)}
\psfrag{B}{(b)}
  \centerline{\includegraphics[height=.2\textheight]{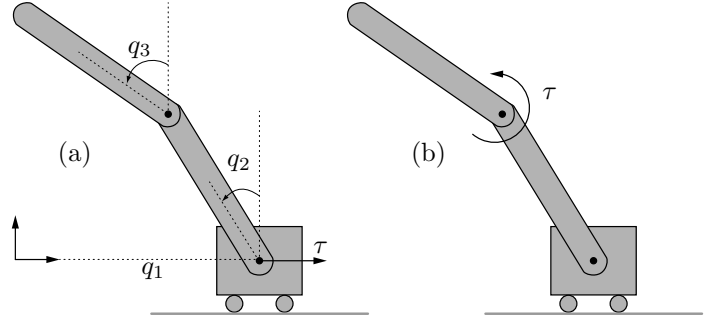}}
  \caption{The double pendulum on a cart of
    Example~\ref{ex:double_pendulum_cart}. Case (a):
    control force on the cart. Case (b): control torque
    on the last joint. The orthogonal frame in the figure is the
    inertial reference frame.}
  \label{fig:double_pendulum_cart}
\end{figure}
Using $q=(q_1,q_2,q_3) \in \Re \times \Se^1 \times \Se^1$ as
generalized coordinates, the Lagrangian of the system is $L(q,\dot q)
= \frac 1 2 g_q(\dot q,\dot q) - P(q)$, with
\[
g_q(\dot q,\dot q) = \dot q\trans \begin{bmatrix}
3 & -2 \cos(q_2) & -\cos(q_3) \\
-2 \cos(q_2) & 2 & \cos(q_2-q_3) \\
-\cos(q_3) & \cos(q_2-q_3) & 1
\end{bmatrix} \dot q,
\]
and
\[
P(q) = ( 2\cos(q_2)+ \cos(q_3)) G,
\]
where $G$ is the gravitational constant. In generalized coordinates,
the control force is the vector $B(q) \tau = [1 \ 0 \ 0]\trans\tau$ (case (a)) or
$[ 0 \ 0 \ 1]\trans \tau$ (case (b)), where $\tau \in \Re$ is the
control input. Letting $\Q = \Re \times \Se^1 \times \Se^1$, we have a
Lagrangian control system $(\Q,g,P,F)$, where $F = dq_1$ or $F=dq_3$,
respectively.

Consider the embedded submanifold of $\Q$, 
\[
\C = \{q \in \Q : q_3 = \rho(q_2)\}, 
\]
where $\rho: \Se^1 \to \Se^1$ is the smooth function
\[
\rho(q_2) = q_2 + 2 \arctan\big((1+ \sqrt 2) \tan(-q_2/2)\big).
\]
The configuration of the double pendulum on $\C$ is illustrated in
Figure~\ref{fig:vhc_double_pendulum}. The function $\rho$ above was
already used in~\cite{ConMagNieTos10} for path following control of a
PVTOL aircraft, and in~\cite{6286994} for pendubot swing-up. In the
context of the double pendulum on a cart of
Figure~\ref{fig:double_pendulum_cart}, the function $\rho$ induces the
interesting property that, on $\C$, the last link does not perform
full revolutions and remains confined to the upper half-plane.

\begin{figure}[htb]
%%   \psfrag{a}{$q_1$}
%%   \psfrag{b}{$q_2$}
%% \psfrag{c}{$q_3$}
%% \psfrag{t}{$\tau$}
  \centerline{\includegraphics[width=.35\textwidth]{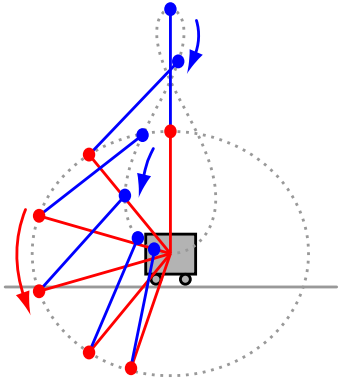}}
  \caption{Configurations of the double pendulum on the VHC $\C$ of
    Example~\ref{ex:double_pendulum_cart}. The missing configurations on the right-hand side are
    deduced by symmetry with respect to the vertical axis.}
  \label{fig:vhc_double_pendulum}
\end{figure}
Letting $h(q) = q_3 - \rho(q_3)$, one may check that, in both cases
(a) and (b), $dh_q D^{-1}(q) B(q) \neq 0$ for all $q \in \C$, so by
the equivalence of Definitions~\ref{defn:regularVHC:coordinates}
and~\ref{defn:regularVHC}, $\C$ is a
regular VHC.  The set $\C$ is diffeomorphic to a cylinder via the
diffeomorphism $\Re \times \Se^1 \to \C$, $(\th^1,\th^2) \mapsto
(\th^1,\th^2,\rho(\th^2))$. Using this global parametrization and the
formulas in~\eqref{eq:Christoffel:induced_connection}, one may show
that, in both cases (a) and (b), the only nonzero Christoffel symbols
of $\nablaC$ are $\GammaC_{22}^1(\th^2)$ and $\GammaC_{22}^2(\th^2)$,
and they are functions of $\th^2$ only. Similarly, the representation
of $\sigma(\grad P)$ in $(\th^1,\th^2)$-coordinates is a function of
$\th^2$ only. Therefore, in both cases (a) and (b) the constrained
dynamics~\eqref{eq:constrained_dynamics:coordinates} have the form
\begin{equation}\label{eq:constrained_dynamics:dpc}
\begin{aligned}
& \ddot \th^1 = - \GammaC_{22}^1(\th^2) (\dot \th^2)^2 - \lambda_1(\th^2) \\
& \ddot \th^2 = - \GammaC_{22}^2(\th^2) (\dot \th^2)^2 - \lambda_2(\th^2),
\end{aligned}
\end{equation}
where $(\lambda_1(\th^2),\lambda_2(\th^2))$ is the coordinates
representation of $\sigma(\grad P)$. The precise expressions are easy
to determine with the formulas in
Section~\ref{sec:constrained_dynamics:coordinates}, but they are too
long to report here. The function $\rho$ is odd, i.e., $\rho(-\th^2) =
- \rho(\th^2)$, and as a result, the functions $\GammaC_{22}^1(\th^2),
\GammaC_{22}^2(\th^2), \lambda_i(\th^2)$ are odd as well.

Now we investigate the Lagrangian nature of the constrained
dynamics. We will show that in case (a) (force on the cart), the
constrained dynamics are {\em not} Lagrangian, while in case (b)
(torque on the second revolute joint), they are. We begin by checking
condition (i) of Theorem~\ref{thm:ILP:general}, i.e., the
metrizability of $\nablaC$. The coefficients of the curvature
endomorphism in $(\th^1,\th^2)$ coordinates may be computed using the
formula~\eqref{eq:curvature_coordinates}. Owing to the fact that only
the symbols $\GammaC_{22}^k$ are nonzero, and that they are functions
of $\th^2$ only, we see from~\eqref{eq:curvature_coordinates} that the
curvature endomorphism is identically zero, i.e., the induced
connection is flat. We can then use
Proposition~\ref{prop:holonomy:flat} to determine whether or not the
induced connection is metrizable. Recall that $(\th^1,\th^2) \in \Re
\times \Se^1$. In what follows, the point $(0,0) \in \Re \times \Se^1$
will be denoted $0$ in subscripts.

The generator of the first homotopy group $\pi_1(\Re \times
\Se^1,(0,0))$ is $[\gamma_0]$, where $\gamma_0: [0,2\pi] \to \Re
\times \Se^1$ is the curve $t \mapsto (0,t \Mod 2\pi)$. In light of
Proposition~\ref{prop:holonomy:flat}, we seek a positive definite
quadratic form that is invariant under the transport map
$\PP_{\gamma_0}$. With reference
to~\eqref{eq:parallel_transport:coordinates}, to find $\PP_{\gamma_0}$
we solve the linear time-varying system
%
%
%
%\begin{equation}\label{eq:parallel_transport:dpc}
\[
\dot X = \begin{bmatrix} 0 \ & -\GammaC_{22}^1(t) \\ 0 \ & -\GammaC_{22}^2(t)
\end{bmatrix} X, \quad X(0) = v,
\]
%\end{equation}
%
%
%
and put $\PP_{\gamma_0}(v) = X(2\pi)$. The solution is $X(t) =
(X^1(t),X^2(t))$ with
\begin{equation}\label{eq:parallel_transport:dpc:solution}
\begin{aligned}
& X^1(t) = v_1 - v_2\int_0^t \GammaC_{22}^1(z) \exp\left(\int_0^z
  -\GammaC_{22}^2(u) du \right) dz  \\
& X^2(t) = v_2 \exp \left(\int_0^t -\GammaC_{22}^2(z) dz \right).
\end{aligned}
\end{equation}
Denote $I_1(t):=-\int_0^t \GammaC_{22}^2(z) dz$ and $I_2(t):=\int_0^t
\GammaC_{22}^1(z) \exp \big( I_1(z) \big) d z$. Then,
\[
\PP_{\gamma_0}(v) = \left[\begin{array}{rr} 1 & -I_2(2\pi) \\ 0 & \exp \big( I_1(2\pi) \big)
\end{array}\right] v.
\]
Recall that the functions $\Gamma_{22}^k$ are odd and
$2\pi$-periodic. The integral over one period of an odd periodic
function is zero.  Using this fact, we have $I_1(2\pi)=0$. Since the
integral of an odd function is an even function, $\exp(I_1(t))$ is
even, and its product with $\GammaC_{22}^1(t)$ is odd.  Thus $I_2(2
\pi)=0$. In conclusion, $\PP_{\gamma_0}$ is the identity map, implying
that any positive definite quadratic form on $( T_{(0,0)} \Re \times
\Se^1) \times ( T_{(0,0)} \Re \times \Se^1)$ is invariant under
$\PP_{\gamma_0}$. By Proposition~\ref{prop:holonomy:flat}, $\nablaC$
is metrizable. This result holds for both cases (a) and (b).

Next, we find all Riemannian metrics on $\Re \times \Se^1$ compatible
with $\nablaC$. Just like in Section~\ref{sec:one_dim_manifolds}, we
will use the notation $g$ (in place of $g_\C$) for such a
metric. Modulo scalar multiples, the generic positive definite
quadratic form on $( T_{(0,0)} \Re \times \Se^1) \times ( T_{(0,0)}
\Re \times \Se^1)$ is
\[
g_0(v_0,w_0) = v_0\trans \begin{bmatrix} 1 & a\\a & b
\end{bmatrix} w_0,
\]
with $a,b \in \Re$ such that $b> a^2$. As in
Section~\ref{sec:one_dim_manifolds}, the Riemannian metric on $\Re
\times \Se^1$ is found by parallel transporting $g_0$ by means
of~\eqref{eq:metric:parallel_translation}. The construction is
illustrated in Figure~\ref{fig:parallel_transport_dpc}.
\begin{figure}[htb]
  \psfrag{t}{$\th^2$}
  \psfrag{T}{$\th^1$}
\psfrag{1}{$\gamma_0^{\bar \th}$}
\psfrag{2}{$\gamma_{\bar \th}^\th$}
\psfrag{A}[c]{$(0,0)$}
\psfrag{B}[c]{$(\th^1,\th^2)$}
\psfrag{C}{$\bar \th$}
\psfrag{v}{$v$}
\psfrag{w}{$\PP_{\gamma^\th_0}(v)$}
  \centerline{\includegraphics[width=.6\textwidth]{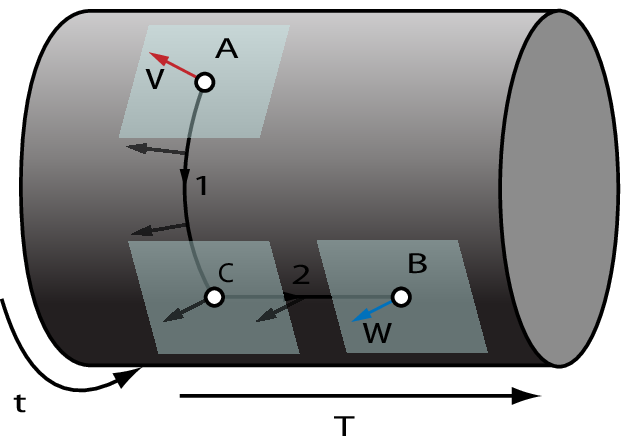}}
  \caption{Parallel transport on $\Re \times \Se^1$ from $(0,0)$ to
    $(\th^1,\th^2)$.}
  \label{fig:parallel_transport_dpc}
\end{figure}

Let $s=(\th^1,\th^2) \in \Re \times \Se^1$ be arbitrary, denote $\bar
\th:=(0,\th^2)$.  Pick any $x^2 \in \pi^{-1}(\th^2)$ and define a path
$\gamma^s_0(t)$ as the concatenation of paths $\gamma^{\bar \th}_0$
and $\gamma^\th_{\bar \th}$, where $\gamma^{\bar \th}_0: [0,x^2] \to
\Re \times \Se^1$ is defined as $\gamma^{\bar \th}_0(t):= (0,t \Mod
2\pi)$, and $\gamma^{\th}_{\bar \th}: [0,\th^1] \to \Re \times \Se^1$
is defined as $\gamma^\th_{\bar \th}(t):= (t,\th^2)$. Then,
$\gamma^\th_0 : [0,\th^1+x^2] \to \Re \times \Se^1$ is a piecewise
smooth path connecting $(0,0)$ to $(\th^1,\th^2)$. See
Figure~\ref{fig:parallel_transport_dpc}. By
Proposition~\ref{prop:parallel_transport:group_property}, we have
$\PP_{\gamma^\th_0} = \PP_{\gamma_0^{\bar \th}} \circ
\PP_{\gamma^\th_{\bar \th}}$. Since $\gamma^{\th}_{\bar \th}$ is a
translation along the real line, $\PP_{\gamma^{\th}_{\bar \th}}$ is
the identity map. On the other hand, $\PP_{\gamma^{\bar\th}_0}$ is
given by~\eqref{eq:parallel_transport:dpc:solution} at time
$\th^2$. In conclusion, the parallel transport map from $(0,0)$ to
$(\th^1,\th^2)$ is
\[
\PP_{\gamma_0^\th} = \left[ \begin{array}{cc} 1 & -I_2(\th^2) \\ 0 & \exp\big(I_1(\th^2))
  \end{array}\right].
\]
By Proposition~\ref{prop:parallel_transport:group_property}, the
parallel transport map from $(\th^1,\th^2)$ to $(0,0)$ is $
\PP_{\gamma^\th_0}^{-1}$. 
Now we define a Riemannian metric $g$ on $\Re
\times \Se^1$ by transporting tangent vectors in $T_{(\th^1,\th^2)}
(\Re \times \Se^1)$ to $T_{(0,0)}(\Re \times \Se^1)$ (cf.~\eqref{eq:metric:parallel_translation}):
\begin{equation}\label{eq:constrained_dynamics:metric:dpc}
\begin{aligned}
g(v_\th,w_\th)&:= 
g_0 \big( \PP_{\gamma^\th_0}^{-1}(v_\th),
\PP_{\gamma^\th_0}^{-1}(w_\th) \big) \\
&=v_\th\trans \left[ \begin{array}{cc} 1 & -I_2(\th^2) \\ 0 &
    \exp\big(I_1(\th^2))
  \end{array}\right]^{-\top} \begin{bmatrix} 1 \ & \ a\\a \ & \ b
\end{bmatrix}  \left[ \begin{array}{rr} 1 & -I_2(\th^2) \\ 0 & \exp\big(I_1(\th^2))
  \end{array}\right]^{-1} w_\th \\
&= v_\th\trans D_\C(\th^2) w_\th,
\end{aligned}
\end{equation}
where
\[
D_\C(\th^2):=\left[\begin{array}{cc} 1 & \exp(-I_1(\th^2))
    (I_2(\th^2)+a) \\ \exp(-I_1(\th^2)) (I_2(\th^2)+a) & \exp(-2
    I_1(\th^2)) (I_2(\th^2) + 2 a I_2(\th^2) + b)
  \end{array}\right].
\]
By this construction, for any $(a,b)$ with $b > a^2$, $g$ is a
Riemannian metric on $\Re \times \Se^1$, and $\nablaC$ is the
Levi-Civita connection of $g$. This result holds for both cases (a) and (b).

Next we turn to condition (ii) of Theorem~\ref{thm:ILP:general},
namely the existence of a smooth function $\PC: \C \to \Re$ such that
$\sigma(\grad P) = \gradC \PC$ or, in $(\th^1,\th^2)$-coordinates, a
function $\PC: \Re \times \Se^1 \to \Re$ such that
\[
\begin{bmatrix} \lambda_1(\th^2) \\ \lambda_2(\th^2) 
\end{bmatrix} = D_\C^{-1}(\th^2) \nabla_\th \PC,
\]
where $\lambda_1,\lambda_2$ are the functions in the constrained
dynamics~\eqref{eq:constrained_dynamics:dpc}. Equivalently, we seek a
function $\PC$ such that 
\begin{equation}\label{eq:one-form}
\begin{aligned}
(d\PC)_\th &= \big[ \lambda_1(\th^2) + \exp(-I_1(\th^2))
    (I_2(\th^2)+a) \lambda_2(\th^2) \big] d\th^1 \\
&+ \big[ \exp(-I_1(\th^2)) (I_2(\th^2)+a) \lambda_1(\th^2) \\
&+ \exp(-2I_1(\th^2))(I_2(\th^2) + 2a I_2(\th^2) +b) \lambda_2(\th)
    \big] d\th^2.
\end{aligned}
\end{equation}
Such a function $\PC$ exists if and only if the one-form on the
right-hand side of the above identity is exact. For this, we need to
check whether this one-form is closed, and whether its integral over
the loop $\gamma_0$ defined earlier is zero. As far as closedness of
the form is concerned, we need to check whether or not there exists $a
\in \Re$ such that
\begin{equation}\label{eq:closed_form}
\partial_{\th^2} \big[\lambda_1(\th^2) + \exp(-I_1(\th^2))
    (I_2(\th^2)+a) \lambda_2(\th^2) \big] =0.
\end{equation}
In case (a) when the control force is on the cart, one can show that
there is no value of $a$ for which~\eqref{eq:closed_form} holds,
whereas in case (b), when there is a control torque on the last
joint,~\eqref{eq:closed_form} is satisfied with $a=-1/2$. In the
latter case, the one-form on the right-hand side
of~\eqref{eq:one-form} is also exact because its components are odd
functions of $\th^2$. We choose any $b > a^2$, for instance $b=1$, and
obtain that the constrained dynamics
in~\eqref{eq:constrained_dynamics:dpc} are a Lagrangian system $(\Re
\times \Se^1,g,\PC)$, with $g$ given
in~\eqref{eq:constrained_dynamics:metric:dpc} and
\[
\PC(\th) = \int_{\gamma_0^\th} d\PC,
\]
with $d\PC$ given in~\eqref{eq:one-form}, and $\gamma^\th_0$ defined
earlier.

We  summarize  our results for this example in the following table.

\medskip
\begin{tabular}{|p{2.5cm}|p{2cm}|p{1.7cm}|p{2cm}|p{2cm}|}
\cline{2-5}
\multicolumn{1}{c|}{}
 & $q_3=\rho(q_2)$ regular? & $\nablaC$  metr'le? &
$\PC$ exists? &  Lagr. exists? \\ 
\cline{2-5} \hline
{\bf  Force on cart} 
& yes & yes & no & no \\
\hline
{\bf  Force on last joint} 
& yes & yes & yes & yes \\
\hline
\end{tabular}
\medskip

Figure~\ref{fig:phase_portrait_s2}  depicts the 
$(\th^2,\dot \th^2)$ orbits  (equivalently,  the $(q_2,\dot q_2)$ orbit) of a few
solutions of the constrained system in cases (a) and (b). In both
cases, we observe two types of behaviours: there are trajectories
along which $q_2$ exhibits a rocking motion around $\pi$, and others
along which $q_2$ performs full revolutions. The behaviour of $q_1$,
not shown in the figure, is a drifting motion with bounded,
sign-definite speed.
\end{example}
\begin{figure}[htb]
\psfrag{x}[c]{$q_2$}
\psfrag{y}[c]{$\dot q_2$}
\psfrag{a}[c]{\small $0$}
\psfrag{b}[c]{\small $\pi/2$}
\psfrag{c}[c]{\small $\pi$}
\psfrag{d}[c]{\small $3\pi/2$}
\psfrag{e}[c]{\small $2\pi$}
\centerline{\includegraphics[width=.45\textwidth]{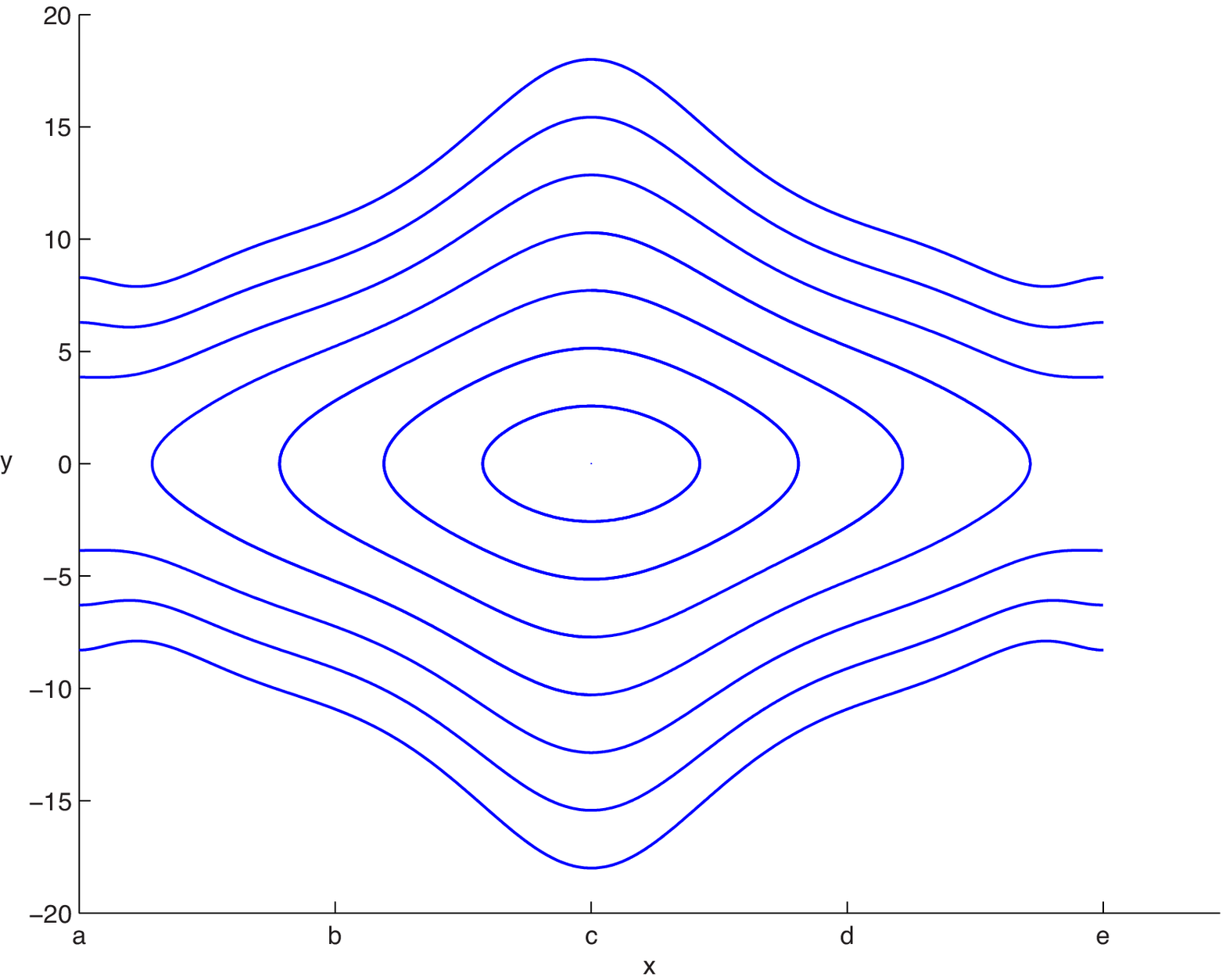}
\includegraphics[width=.45\textwidth]{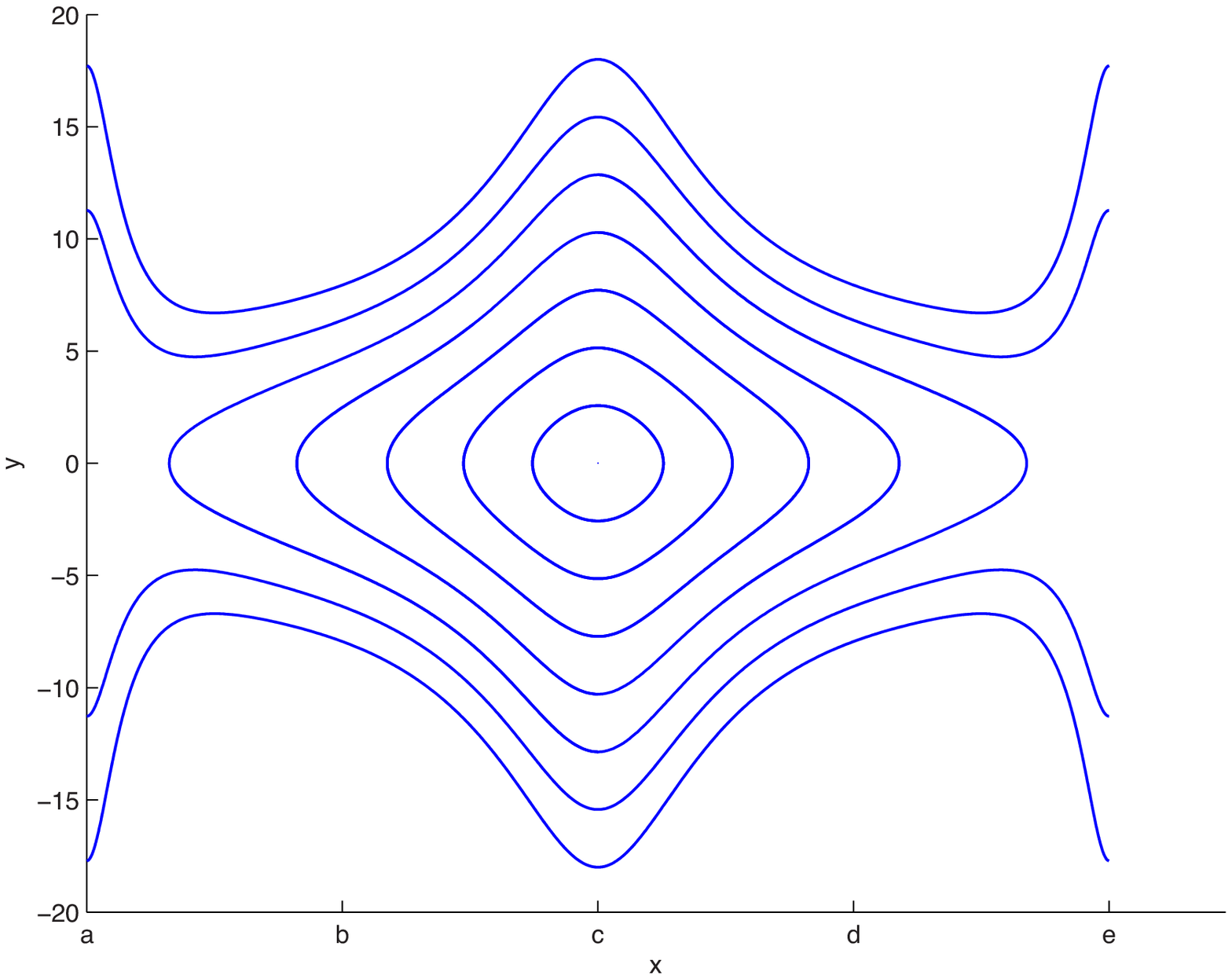}}
\caption{$(q_2,\dot q_2)$ orbits of a few
  solutions of the double pendulum on a cart subject to the VHC $q_3
  = \rho(q_2)$. On the left, case (a) (force on cart). On the right,
  case (b) (torque on last joint).}
\label{fig:phase_portrait_s2}
\end{figure}
\section{Conclusions} 
We introduced a coordinate-free framework of virtual holonomic
constraints for underactuated Lagrangian control systems, exposing the
role of induced connections in the characterization of constrained
dynamics. In this framework, the classical mechanics notion of ideal
holonomic constraint becomes the special case in which the
acceleration distribution is orthogonal to the VHC. We showed that,
generally, the constrained dynamics are not Lagrangian, and the
metrizability of the induced connection is key for the existence of a
Lagrangian structure.  When the constrained dynamics are forced (i.e.,
when the order of the regular VHC is less than the number of control
inputs), the problem remains open of determining when the constrained
dynamics are feedback equivalent to a Lagrangian control system. One
possible avenue of investigation for the solution of this latter
problem is to globalize the local theory of~\cite{RicRes10} in the
context of affine connection control systems.

\appendix

\section{Proof of Lemma~\ref{lem:controlled_invariance}}

We begin by observing that if $X \in \X(\Q)$, then for each $Y_q \in
T \Q$, $\vlft(X)(Y_q) \in \Ker d \pi_{Y_q}$. Indeed, 
\[
\begin{aligned}
d \pi_{Y_q} \big( \vlft(X)(Y_q)\big) &= d \pi_{Y_q} \big( (d/dt)|_{t=0}
(Y_q + t X(q)) \big) \\
&= \frac d {dt} \Big|_{t=0} \pi(Y_q +  t X(q) ) 
=\frac d {dt} \Big|_{t=0} q =0.
\end{aligned}
\]

Using the above and the property of geodesic sprays that $d\pi_{X_q}
(S(X_q)) = X_q$, for each $X_q \in T\C$ we have
\begin{equation}\label{eq:drift}
\begin{aligned}
d \pi_{X_q} \big( S(X_q) - \vlft(\grad P)(X_q) \big) &= d \pi_{X_q}
\big( S(X_q) ) - d \pi_{X_q} \big( \vlft(\grad P)(X_q) \big) \\
& = X_q \in T_q \C.
\end{aligned}
\end{equation}
From~\eqref{eq:drift} we deduce that
\begin{equation}\label{eq:drift2}
(\forall X_q \in T\C) \ S(X_q)- \vlft(\grad P)(X_q) \in ( d \pi_{X_q}
  )^{-1}(T_q \C),
\end{equation}
thus the proof of the lemma will be complete if we show that
\begin{equation}\label{eq:dpi_inverse_decomposition}
(\forall X_q \in T\C)\ ( d \pi_{X_q} )^{-1}(T_q \C) = T_{X_q} T\C
  \oplus \vlft(\D_A)(X_q).
\end{equation}
Let $X_q \in T\C$ be arbitrary. Since $\dim ( (d
\pi_{X_q})^{-1}(T_q \C)) = 2 n - m = (2n - 2m) + m = \dim( T_{X_q} T\C
) + \dim ( \vlft(\D_A)(X_q) )$, to
prove~\eqref{eq:dpi_inverse_decomposition} we need to show  that the
subspaces $T_{X_q} T\C$ and $\vlft(\D_A)(X_q)$ are independent and
contained in the subspace $( d \pi_{X_q} )^{-1}(T_q \C)$.

First we show that $T_{X_q} T\C \subset ( d
\pi_{X_q})^{-1}(T_q \C)$.
Let $Y_{X_q} \in T_{X_q} T\C$, then there exists a smooth curve in
$T\C$, $\gamma: \Re \to T\C$, such that $\gamma(0)=X_q$ and $\dot
\gamma(0)=Y_{X_q}$. We have $d\pi_{X_q} ( Y_{X_q} ) = (d/dt)|_{t=0}
\pi(\gamma(t))$. Since $\gamma$ is a curve on $T \C$, $\pi(\gamma(t))$
is a curve on $\C$, and thus $(d/dt)|_{t=0}
\pi(\gamma(t)) \in T_q \C$, which proves that $Y_{X_q} \in ( d
\pi_{X_q})^{-1}(T_q \C)$.

Next we show that $\vlft(\D_A)(X_q) \subset ( d \pi_{X_q})^{-1}(T_q
\C)$. This follows directly from the fact that $\vlft(\D_A)(X_q) \in
\Ker d\pi_{X_q} \subset ( d \pi_{X_q})^{-1}(T_q \C)$.

Finally, let $Y_{X_q} \in T_{X_q} T\C \cap \vlft(\D_A)(X_q)$. Then
there exists $F_q \in \D_A(q)$ such that $Y_{X_q} = (d/dt)|_{t=0}
(X_q+ t F_q)$. Moreover, $(d/dt)|_{t=0} (X_q+ t F_q) \in T_{X_q} T\C$,
which implies that $F_q \in T_q \C$. Since $F_q \in \D_A(q)$, the
regularity condition~\eqref{eq:transversality_condition} implies that
$F_q=0$, and therefore $Y_{X_q}=0$. Thus the subspaces $T_{X_q} T\C$
and $\vlft(\D_A)(X_q)$ are independent, which shows
that~\eqref{eq:dpi_inverse_decomposition}
holds. Together,~\eqref{eq:drift2}
and~\eqref{eq:dpi_inverse_decomposition} prove the lemma.
\qed

\bibliographystyle{AIMS} 
\bibliography{biblio}
\end{document}